\documentclass[preprint,compress,12pt]{elsarticle}
\usepackage{graphicx}
\usepackage{pifont,latexsym,ifthen,amsthm,rotating,calc,textcase,booktabs}
\usepackage{amsfonts,amssymb,amsbsy,amsmath}
\usepackage{mathrsfs}
\usepackage{appendix}
\newtheorem{theorem}{Theorem}[section]
\newtheorem{lemma}[theorem]{Lemma}

\newtheorem{remark}[theorem]{Remark}
\newtheorem{proposition}[theorem]{Proposition}

\usepackage{natbib}
\newtheorem{rhp}{RH Problem}
\numberwithin{equation}{section}
\usepackage{todonotes}
\usepackage{float}
\usepackage{subfigure}
\usepackage{enumitem}
\usepackage[colorlinks,linkcolor=blue]{hyperref}
\hypersetup{
    colorlinks=true, 
    linktoc=all,     
    linkcolor=blue,  
    citecolor=blue
}

\begin{document}

\begin{frontmatter}
\title{$L^{2}$-Sobolev space bijectivity and existence of global solutions for the  matrix nonlinear Schr\"{o}dinger equations}

\author[inst2]{Yuan Li}
\author[inst2]{Xinhan Liu}
\author[inst2]{Engui Fan$^{*,}$}

\address[inst2]{ School of Mathematical Sciences and Key Laboratory of Mathematics  for Nonlinear Science, Fudan University, Shanghai, 200433, China\\
* Corresponding author and e-mail address: faneg@fudan.edu.cn  }

\begin{abstract}
We consider  the Cauchy problem to the general defocusing and focusing $p\times q$ matrix nonlinear Schr\"{o}dinger (NLS) equations  with initial data
 allowing arbitrary-order poles and spectral singularities. By establishing the $L^{2}$-Sobolev space bijectivity of the direct and inverse scattering transforms associated with a $(p+q)\times(p+q)$ matrix spectral problem, we prove that both  defocusing and focusing  matrix NLS equations are globally well-posed   in the weighted Sobolev space $H^{1,1}(\mathbb{R})$.
\end{abstract}

\end{frontmatter}
\tableofcontents

\section{Introduction}
In this work, by establishing  the $L^{2}$-Sobolev space bijectivity of the direct and inverse scattering transforms associated with the $(p+q)\times(p+q)$ matrix spectral problem (\ref{2.1}),
 we further prove the  existence of global solutions to  the Cauchy problem for  the following  general $p\times q$  
  matrix nonlinear Schr\"{o}dinger (NLS)  equation  \cite{Tsu,Zak}
 \begin{align}
&i\boldsymbol{Q}_{t}+\boldsymbol{Q}_{xx}-2\sigma \boldsymbol{Q}\boldsymbol{Q}^{\dag}\boldsymbol{Q}=0_{p\times q},\quad \sigma=\pm1,\quad x\in\mathbb{R},\label{mnls}\\
&\boldsymbol{Q}(x,0)=\boldsymbol{Q}_{0}(x),\label{mnls1}
\end{align}
where $\boldsymbol{Q}(x,t)$ is a $p\times q$ matrix valued function, $0_{p\times q}$ is the $p\times q$ zero matrix and the sign  $\dag$ denotes Hermitian conjugate. The choice $\sigma=1$ and $\sigma=-1$ corresponding to the defocusing and focusing matrix NLS equations, respectively.

The   matrix NLS equation  \eqref{mnls} is a  natural matrix generalization   of the well-known scalar NLS equation
$$ iq_t +q_{xx}-2\sigma|q|^2q=0,\quad \sigma=\pm1. $$
 Ablowitz et al. \cite{Abl} detailedly described matrix generalization of the inverse scattering transform (IST) and dynamics of the soliton interactions for the $p\times q$ matrix NLS equation  \eqref{mnls}. A class of explicit solutions for the focusing case of \eqref{mnls} were derived by Demontis and Mee \cite{Dem}. When $p=q$, the $p\times p$ matrix NLS equation under the vanishing boundary conditions has an infinite number of conservation laws and multi-soliton solutions \cite{Tsu}. For the $p\times p$ square matrix NLS equation  with nonvanishing boundary conditions, Ieda et al. \cite{Ieda2} considered the IST on a two-sheeted Riemann surface and cut plane, Prinari et al. \cite{Pri1} advanced this result and developed the IST on the standard complex plane.
  We have to emphasize that if $p=q=2$ and $\boldsymbol{Q}(x,t)$ is a symmetric matrix defined by
\begin{equation*}
\boldsymbol{Q}(x,t)=\begin{bmatrix}q_{1}(x,t)&q_{0}(x,t)\\q_{0}(x,t)&q_{-1}(x,t)\end{bmatrix},
\end{equation*}
the matrix NLS equation  \eqref{mnls} reduces  to the multi-component Gross-Pitaevskii (GP) equation
\begin{equation}\label{GP}
\left\{\begin{aligned}
&i\partial_{t}q_{1}+\partial_{x}^{2}q_{1}-2\sigma q_{1}\left(|q_{1}|^{2}+2|q_{0}|^{2}\right)-2\sigma q_{0}^{2}q_{-1}^{*}=0,\\
&i\partial_{t}q_{-1}+\partial_{x}^{2}q_{-1}-2\sigma q_{-1}\left(|q_{-1}|^{2}+2|q_{0}|^{2}\right)-2\sigma q_{0}^{2}q_{1}^{*}=0,\quad \sigma=\pm1, \\
&i\partial_{t}q_{0}+\partial_{x}^{2}q_{0}-2\sigma q_{0}\left(|q_{1}|^{2}+|q_{0}|^{2}+|q_{-1}|^{2}\right)-2\sigma q_{1}q_{0}^{*}q_{-1}=0,
\end{aligned}
\right.
\end{equation}
which can be use to describe the  Bose-Einstein condensates \cite{Ieda,Kev}.
When $\sigma=1$, the interatomic interactions are repulsive and the spin-exchange interactions are anti-ferromagnetic \cite{Uch}, while $\sigma=-1$  corresponds to attractive interatomic interactions and ferromagnetic spin-exchange interactions \cite{Kur,LiST}.
In recent years, Prinari et al. \cite{Pri2} investigated the soliton solutions and their properties   for the equations \eqref{GP}  with  the IST method.
Yan \cite{Yan1,Yan2} extended  the Fokas unified method  to investigate the initial-boundary value problem for equation  \eqref{GP}.  The long-time asymptotics for the solution to  the  Cauchy problem of the  GP equation  \eqref{GP} in the focusing case $\sigma=-1$   was  obtained by Geng et al. \cite{Geng}.

The $p\times q$ matrix NLS equations \eqref{mnls} admit the following  $(p+q)\times(p+q)$  AKNS-type    spectral problem      \cite{Tsu}
\begin{equation}
\psi_{x}=-ik\sigma_{3}\psi+U\psi, \label{2.1}
\end{equation}
where
\begin{equation}\label{Q}
\sigma_{3}=\begin{bmatrix}I_{p}&0_{p\times q}\\0_{q\times p}&-I_{q}\end{bmatrix},\quad
U=\begin{bmatrix}0_{p\times p}&\boldsymbol{Q}\\ \sigma \boldsymbol{Q}^{\dag}&0_{q\times q}\end{bmatrix},\quad \sigma=\pm1.
\end{equation}
Here $k\in\mathbb{C}$ is a spectral parameter, $I_{p}$ and $I_{q}$ are the $p\times p$ and $q\times q$ identity matrix, respectively. In 1998, Zhou \cite{Zhou1998} first established the $L^{2}$-Sobolev space bijectivity for the scattering-inverse scattering transforms for the case $p=q=1$. Later, Liu \cite{LiuJQ} extended Zhou's bijectivity result to the case   $p=q=2$, which is the first step towards establishing the bijectivity of the scattering-inverse scattering transforms for the more general $(p+q)\times(p+q)$  spectral problem   \eqref{2.1}.  For higher order matrix  spectral problem, the scattering-inverse scattering analysis will become very complicated. To the best of our knowledge, there still  has been  no  work  on  the bijectivity of the scattering-inverse scattering transforms
 for  the    spectral problem    \eqref{2.1} so far.

The   purpose of our  work is to generalize the results of Zhou and Liu  \cite{Zhou1998,LiuJQ} to the  spectral problem  \eqref{2.1}.
We  shall  establish  the bijectivity of the scattering-inverse scattering transforms, further prove global well-posedness  to the $p\times q$ matrix NLS equations \eqref{mnls} for the initial data $\boldsymbol{Q}_{0}(x)$  allowing arbitrary-order poles and spectral singularities.

\subsection{Main results}
Our main results are the following:
\begin{theorem}\label{thm}  Let $\boldsymbol{Q}$ be the potential of the spectral problem  \eqref{2.1},  
$v_{\pm}(k)$ and $V_{\pm}(k)$ be the scattering data  associated with $\boldsymbol{Q}$ corresponding  to  the defocusing  and focusing cases, respectively.
 Then we have 
\begin{enumerate}[label=(\roman*)]
  \item For defocusing case ($\sigma=1$), the direct scattering and inverse scattering maps
  \begin{align}
  \mathcal{D}: \ &H^{1,1}(\mathbb{R})\ni \boldsymbol{Q}\longmapsto v_{\pm}-I_{p+q}\in H^{1,1}(\mathbb{R})\label{D11}\\
  \mathcal{I}: \ &H^{1,1}(\mathbb{R})\ni v_{\pm}-I_{p+q} \longmapsto \boldsymbol{Q}\in H^{1,1}(\mathbb{R})\label{I11}
  \end{align}
  are Lipschitz continuous.
  \item For focusing case ($\sigma=-1$), the direct scattering and inverse scattering maps
  \begin{align}
  \mathcal{D}: \ &H^{1,1}(\mathbb{R})\ni \boldsymbol{Q}\longmapsto V_{\pm}-I_{p+q}\in H^{1,1}(\partial\Omega_{\pm})\label{D22}\\
  \mathcal{I}: \ &H^{1,1}(\partial\Omega_{\pm})\ni V_{\pm}-I_{p+q} \longmapsto \boldsymbol{Q}\in H^{1,1}(\mathbb{R})\label{I22}
  \end{align}
  are Lipschitz continuous.
\end{enumerate}
Moreover, given initial data  $\boldsymbol{Q}_{0}\in H^{1,1}(\mathbb{R})$ without any spectral restrictions, there exists a unique solution $\boldsymbol{Q}\in C\left([-T,T], H^{1,1}(\mathbb{R})\right)$ of the Cauchy problem \eqref{mnls}-\eqref{mnls1} for every $T>0$. The map $\boldsymbol{Q}_{0}\mapsto \boldsymbol{Q}$ is Lipschitz continuous from $H^{1,1}(\mathbb{R})$ to $C\left([-T,T], H^{1,1}(\mathbb{R})\right)$ for every $T>0$.
\end{theorem}
\begin{remark}
The main difference between the defocusing ($\sigma=1$) and focusing ($\sigma=-1$) matrix NLS equations is that the scattering date need to be considered in different spaces
$H^{1,1}(\mathbb{R})$ defined in \eqref{sob} and $H^{1,1}(\partial\Omega_{\pm})$ defined in \eqref{space}. This is mainly determined by whether the jump relation is on a contour which has self-intersections.
\end{remark}
\begin{remark}
The space $H^{1,1}(\mathbb{R})$ and $H^{1,1}(\partial\Omega_{\pm})$ are determined by the $t$ Lax equation of the matrix NLS equation  \eqref{mnls} as follows
\begin{equation}\label{2.2}
\psi_{t}=-2ik^{2}\sigma_{3}\psi+P\psi,\quad
P=2kU+i\sigma_{3}U_{x}-i\sigma_{3}U^{2}.
\end{equation}
\end{remark}

The proof of Theorem \ref{thm} mainly rely on the IST method in the form of the Riemann-Hilbert (RH) problem and Zhou's approach \cite{Zhou1989b,Zhou1989a,Zhou1998}. As is well known, the IST method is an extremely classical approach in integrable systems, and in the last decade it has taken the powerful advantages in dealing with the well-posedness of integrable equations \cite{Bah1,Jen1,Jen2,Jenkins,LiuJQ,Liu16,Pel1,Peli}. It is worth mentioning that Jenkins et al.\cite{Jenkins} first applied Zhou's techniques \cite{Zhou1989b,Zhou1989a,Zhou1998} to solve the global well-posedness problems for the derivative NLS  equation with arbitrary singularities. Inspired by the idea of Jenkins et al. \cite{Jenkins} and based on the corresponding bijectivity results, we also give a well-posedness result for the matrix NLS equation  \eqref{mnls} in Theorem \ref{thm}. This further supports the opinion of Jenkins et al. \cite{Jenkins} that ``Zhou's methods are quite general and are likely applicable to well-posedness questions for other integrable PDEs in one space dimension.''

The $L^{2}$-Sobolev space bijectivity of the inverse scattering transform for the $(p+q)\times(p+q)$ matrix spectral problem is hard to establish because which involves a lot of matrix estimates and matrix calculations. The first obstacle is how to define the appropriate scattering data. If the scattering data are defined as scalar form  like   $2\times2$ or  $3\times3$ case, then we would be dealing with relations between $(p+q)^{2}$ scattering data, the difficulties of which are unimaginable. An efficient way to do this is to write the $(p+q)\times(p+q)$ matrix in the form of a $2\times2$ block matrix. However in this case the scattering data are matrix valued functions so how to guarantee the invertibility of the scattering data $\boldsymbol{A}(k), \boldsymbol{D}(k)$ and build the unique solvable RH problem becomes another obstacle. We discuss these problems in two cases. (i) For the defocusing case, we find that the matrix scattering data $\boldsymbol{A}(k)$ and $\boldsymbol{D}(k)$ are naturally invertible by using the properties of self-adjoint operator and the symmetry of the block scattering matrix $S(k)$. Not only that, but the latter makes the jump matrix $v(k)$ satisfy the conditions of vanishing lemma \cite{Zhou1989b}.
(ii) For the focusing case, the corresponding problems become very complicated. Unlike the defocusing case, $\boldsymbol{A}(k)$ and $\boldsymbol{D}(k)$ are not naturally invertible so we use Zhou's techniques to define a new function   $\boldsymbol{M} (x,k)$ to avoid any spectral restriction on the initial data. However, the computation of the jump matrix $V(k)$ along the contour $\Gamma$ (Figure \ref{f1}) is tricky since it involves the inverse of the $(p+q)\times(p+q)$ matrix, which can neither be written directly (like $2\times2$ matrix) nor with the help of the adjoint matrix (like $3\times3$ matrix). Fortunately, we find that we can overcome this difficulty by treating the $(p+q)\times(p+q)$ matrix as a block matrix and taking full advantage of the relation and symmetry of the eigenfunctions and the block scattering matrix. In addition, the contour $\Gamma$ has self-intersections therefore we have to verify the matching conditions at the intersections involving the high-order triangular matrix $V_{\pm}(k)$.

\subsection{Framework of the proof}
This work is organized as follows.

The aims of Section \ref{s2} are to solve the direct scattering problems associated with the defocusing and focusing $p\times q$ matrix NLS equations. Subsection \ref{s21} reports the properties of eigenfunctions $m^{\pm}$ associated with a $(p+q)\times(p+q)$ matrix spectral problem. The regularity and Lipschitz continuity of the matrix scattering data and reflection coefficient are given in Subsection \ref{s22}.

In Section \ref{s3}, we focus on inverse scattering  results and well-posedness analysis for the defocusing   matrix NLS equation. We first give the  symmetry properties of the eigenfunctions $m^{\pm}$ and the block scattering matrix $S$ in Subsection \ref{s31}. Then we use the properties of matrix scattering date establish a well-defined RH problem  in Subsection \ref{s32}. Subsection \ref{s33} gives the time evolution of the matrix scattering data. To prove the Lipschitz continuity of the inverse scattering map, we consider the solvability of RH problem with in a larger space $H^{\frac{1}{2}+\varepsilon}(\mathbb{R}), \varepsilon>0$ and obtain the uniform resolvent estimates in Subsection \ref{s34}. With the help of Cauchy integral operator $\mathcal{C}_{\mathbb{R}}^{\pm}$, the decay and smoothness properties of reconstructed potential are presented in Subsection \ref{s35}. The proof of Theorem \ref{thm} for defocusing matrix NLS equation are given in Subsection \ref{s36}.

Section \ref{s4} reports the inverse scattering results for the focusing  matrix NLS equation by using a similar way as Section \ref{s3}. In order to overcome the difficulties posed by poles and spectral singularities, we use the properties of scattering date shown in Subsection \ref{s41} to construct an   RH problem  and calculate the jump matrices along the self-intersecting contour $\Gamma$ (see Figure \ref{f1}) in Subsection \ref{s42}.
Then we verify the matching conditions at the intersections involving the high-order triangular matrix $V_{\pm}(k)$. The time evolution of the matrix scattering data and the solvability of RH problem with self-intersecting contours in a larger space $H^{\frac{1}{2}+\varepsilon}(\Gamma), \varepsilon>0$ are presented in Subsections \ref{s43} and \ref{s44}. In Subsection \ref{s45}, we define the new contour $\widehat{\Gamma}$ (see Figure \ref{f3}) and give the estimates of the reconstructed potential using the Cauchy integral operator $\mathcal{C}_{\widehat{\Gamma}}^{\pm}$. Finally, the proof of Theorem \ref{thm} for focusing matrix NLS equation are given in Subsection \ref{s46}.

\subsection{Notations}
Let us now recall the definition of the weighted Sobolev space mentioned above. The weighted Sobolev space $H^{k,j}(\mathbb{R})$ defined by
\begin{equation}\label{sob}
H^{k,j}(\mathbb{R})=H^{k}(\mathbb{R})\cap L^{2,j}(\mathbb{R}),
\end{equation}
where $H^{k}(\mathbb{R})$ denotes the usual Sobolev space with the norm
\begin{equation*}
\|f\|_{H^{k}}:=\|\langle \cdot\rangle^{k} \hat{f}\|_{L^{2}}, \quad \langle \cdot \rangle=\sqrt{1+|\cdot|^2},
\end{equation*}
here $\hat{f}$ is the Fourier transform of $f$,
and $L^{2,j}(\mathbb{R})$ is the weighted $L^{2}(\mathbb{R})$ space with the norm
\begin{equation*}
\|f\|_{L^{2,j}}:=\|\langle \cdot\rangle^{j} f\|_{L^{2}}.
\end{equation*}
The $L^{p}$ matrix norm of the matrix function $U$ are denoted as
\begin{equation*}
\|U\|_{L^{p}}:=\left\| \left|U\right| \right\|_{L^{p}},
\end{equation*}
where $|U|=\sqrt{\operatorname{tr}(U^{\dag}U)}$.

Next, we describe the definition of $H^{k}(\partial\Omega)$ introduced by Beals and Coifman \cite{Beals}, which also can be found in \cite{Zhou1989b}.
\begin{enumerate}[label=(\roman*)]
  \item If $\Omega$ is a connected open region and $\partial\Omega$ is the nonsmooth boundary of $\Omega$, then $f\in H^{k}(\partial\Omega)$ means that $f$ is $ H^{k}$ on each smooth piece and $f$ matches from two sides to the order $k-1$ at each nonsmooth point.
  \item If $\Omega$ is a disconnected region, then $f\in H^{k}(\partial\Omega)$ implies that $f\in H^{k}(\partial\Omega_{i})$ for every connected component $\Omega_{i}$ of $\Omega$.
  \item The weighted Sobolev space $H^{k,j}(\partial\Omega)$ is defined by
      \begin{equation}\label{space}
      H^{k,j}(\partial\Omega)=H^{k}(\partial\Omega)\cap L^{2,j}(\partial\Omega).
      \end{equation}
      The corresponding Sobolev norms are defined in the obvious way.
\end{enumerate}

\section{Direct Scattering Map}\label{s2}
In this section, we state some main results on the direct scattering transform associated with the Cauchy problem \eqref{mnls}.

\subsection{$(p+q)\times(p+q)$ matrix spectral problem}\label{s21}
Let us rewrite the system \eqref{2.1} in the following equivalent form
\begin{equation}\label{2.3}
m_{x}=-ik\operatorname{ad}\sigma_{3}(m)+Um,
\end{equation}
where
\begin{equation*}
m(x,k)=\psi(x,k)e^{ixk\sigma_{3}},
\end{equation*}
and $\operatorname{ad}\sigma_{3}(m)=[\sigma_{3},m]=\sigma_{3}m-m\sigma_{3}$.
This system can be further rewritten in the form of an integral equation
\begin{align}\label{2.4}
m(x,k)=e^{-ixk\operatorname{ad}\sigma_{3}}\beta(k)+\int_{\delta}^{x}e^{i(y-x)k\operatorname{ad}\sigma_{3}}U(y)m(y,k)dy,
\end{align}
where $e^{-ixk\operatorname{ad}\sigma_{3}}\beta(k)=e^{-ixk\sigma_{3}}\beta(k) e^{ixk\sigma_{3}}$. The matrix function $\beta(k)$ is independent of $x$, and the lower limit $\delta$ can be chosen differently for the different entries of the matrix.

If $\boldsymbol{Q}\in L^{1}(\mathbb{R})$, then for every $k\in\mathbb{R}$, there exist unique bounded solutions $m^{\pm}(x,k)$ of \eqref{2.3} satisfying the boundary conditions
\begin{equation}\label{mi}
m^{\pm}(x,k)\rightarrow I_{p+q},\quad x\rightarrow\pm\infty,
\end{equation}
here $I_{p+q}$ is the $(p+q)\times(p+q)$ identity matrix. Moreover, the bounded solutions $m^{\pm}(x,k)$ can be expressed by the Volterra integral equations
\begin{equation}\label{2.5}
m^{\pm}(x,k)=I_{p+q}+\int_{\pm\infty}^{x}e^{i(y-x)k\operatorname{ad}\sigma_{3}}U(y)m^{\pm}(y,k)dy,
\end{equation}
and satisfy $\sup_{k\in\mathbb{R}}\|m^{\pm}(x,k)\|_{L_{x}^{\infty}(\mathbb{R})}<\infty$.
It is more necessary to write the modified eigenfunctions $m^{\pm}$ in the form of the block matrix, so we set
\begin{equation*}
m^{\pm}=\begin{bmatrix}m_{11}^{\pm}&m_{12}^{\pm}\\m_{21}^{\pm}&m_{22}^{\pm}\end{bmatrix},
\end{equation*}
where $m_{11}^{\pm}$, $m_{12}^{\pm}$, $m_{21}^{\pm}$ and $m_{22}^{\pm}$ are the $p\times p$, $p\times q$, $q\times p$ and $q\times q$ matrices, respectively. All $(p+q)\times(p+q)$ matrices in this work are written as $2\times2$ block matrices in this way. In addition, we denote
\begin{equation*}
\begin{bmatrix}m_{11}^{\pm}\\m_{21}^{\pm}\end{bmatrix}= m_{1}^{\pm}\quad \text{and}\quad \begin{bmatrix}m_{12}^{\pm}\\m_{22}^{\pm}\end{bmatrix}= m_{2}^{\pm},
\end{equation*}
that is, $m_{1}^{\pm}$ and $m_{2}^{\pm}$ are $(p+q)\times p$ and $(p+q)\times q$ matrices, respectively. For all $x\in\mathbb{R}$, $m_{1}^{-}(x,k)$ and $m_{2}^{+}(x,k)$ are analytic for $\text{Im}k\geq0$, while $m_{2}^{-}(x,k)$ and $m_{1}^{+}(x,k)$ are analytic for $\text{Im}k\leq0$. The above result for $m^{\pm}$ can be derived by constructing the Neumann series.

To explore more properties of $m^{\pm}$, we define the integral operator $K$ by
\begin{equation}\label{2+kk}
(Kf)(x,k):=\int_{-\infty}^{x}e^{i(y-x)k\operatorname{ad}\sigma_{3}}U(y)f(y,k)dy,
\end{equation}
then the integral equation \eqref{2.5} for $m^{-}$ can be rewritten
in the operator form
\begin{equation}\label{2+mo}
m^{-}(x,k)=I_{p+q}+Km^{-}(x,k),
\end{equation}
furthermore, we have
\begin{equation}\label{2+mk}
m^{-}(x,k)-I_{p+q}=(I-K)^{-1}KI_{p+q}(x,k).
\end{equation}
The existence of the inverse operator $(I-K)^{-1}$ is guaranteed by the following Lemma \ref{lk}.
\begin{lemma}\label{lk}
If $\boldsymbol{Q}\in L^{1}(\mathbb{R})$, then the integral operator $I-K$ is invertible in the space $L_{x}^{\infty}(\mathbb{R}, L_{k}^{2}(\mathbb{R}))$ and satisfies the following estimate
\begin{align}\label{2+k}
\left\|(I-K)^{-1}\right\|_{L_x^{\infty}L_{k}^{2}\rightarrow L_x^{\infty}L_{k}^{2}}\leq e^{\|U\|_{L^1}}.
\end{align}
Moreover, if $\boldsymbol{Q}\in L^{2}(\mathbb{R})$, then $(KI_{p+q})(x,k)\in L_{x}^{\infty}(\mathbb{R}, L_{k}^{2}(\mathbb{R}))$ and satisfies
\begin{align}\label{2+pro1}
\left\|(KI_{p+q})(x,k)\right\|_{L_{x}^{\infty}L_{k}^{2}}=\left\|\int_{-\infty}^{x}e^{i(y-x)k\operatorname{ad}\sigma_{3}}U(y)dy\right\|_{L_{x}^{\infty}L_{k}^{2}}\lesssim\left\|U\right\|_{L^{2}}.
\end{align}
These two estimates also hold in the space $L_{x}^{\infty}(\mathbb{R}^{-}, L_{k}^{2}(\mathbb{R}))$.
\end{lemma}
\begin{proof}
Let $f(x,k)$ be a matrix function belonging to the space  $L_{x}^{\infty}(\mathbb{R};L_{k}^{2}(\mathbb{R}))$ with the norm
\begin{align*}
\left\|f(x,k)\right\|_{L_{x}^{\infty}L_{k}^{2}}=\sup_{x\in\mathbb{R}}\left\|f(x,k)\right\|_{L_{k}^{2}(\mathbb{R})},
\end{align*}
for every $x\in \mathbb{R}$, an application of Minkowski inequality gives
\begin{align*}
\left\|(Kf)(x,k)\right\|_{L_{k}^{2}}&\leq\int_{-\infty}^{x}|U(y)|\left\|f(y,k)\right\|_{L_{k}^{2}}dy\\
&\leq\sup_{y\in(-\infty,x)}\left\|f(y,k)\right\|_{L_{k}^{2}}\int_{-\infty}^{x}|U(y)|dy.
\end{align*}
By iteration we can further obtain
\begin{align*}
\left\|(K^{n}f)(x,k)\right\|_{L_{k}^{2}}\leq\frac{\left\|U\right\|_{L^{1}}^{n}}{n!}\left\|f(x,k)\right\|_{L_{x}^{\infty}L_{k}^{2}},
\end{align*}
which implies that the operator $I-K$ is invertible in the space $L_{x}^{\infty}(\mathbb{R}, L_{k}^{2}(\mathbb{R}))$ and \eqref{2+k} holds.

It follows from Plancherel formula that $\forall$ $x\in\mathbb{R}$,
\begin{align}\label{2+pro1p}
\begin{aligned}
\left\|(KI_{p+q})(x,k)\right\|_{L_{k}^{2}}=&\left\|\int_{-\infty}^{x}e^{i(y-x)k\operatorname{ad}\sigma_{3}}U(y)dy\right\|_{L_{k}^{2}}\\
\lesssim&\left(\int_{-\infty}^{x}|U(y)|^{2}dy\right)^{\frac{1}{2}}
\leq\left\|U\right\|_{L^{2}},
\end{aligned}
\end{align}
thus, the estimate \eqref{2+pro1} holds. The proof of \eqref{2+pro1} can also be found in \cite{Peli}.
\end{proof}

Making full use of the properties of the integral operator $K$ and equality \eqref{2+mk}, we can obtain some significant properties of $m^{\pm}(x,k)-I_{p+q}$.
\begin{proposition}\label{lm}
If $\boldsymbol{Q}\in H^{1,1}(\mathbb{R})$, then $m^{\pm}(x,k)-I_{p+q}\in L_{x}^{\infty}(\mathbb{R}^{\pm}, H_{k}^{1}(\mathbb{R}))$ and the map $\boldsymbol{Q}\mapsto [m^{\pm}-I_{p+q}]$ is Lipschitz continuous.
\end{proposition}
\begin{proof}
The proof is given in terms of $m^{-}(x,k)$. First, we proof that $m^{-}(x,k)-I_{p+q}\in L_{x}^{\infty}(\mathbb{R}^{-}, L_{k}^{2}(\mathbb{R}))$. According to \eqref{2+mk} and the result of Lemma \ref{lk}, we have
\begin{equation}\label{2+m-i}
\begin{aligned}
\left\|m^{-}(x,k)-I_{p+q}\right\|_{L_{x}^{\infty}L_{k}^{2}}=&\left\|(I-K)^{-1}KI_{p+q}(x,k)\right\|_{L_{x}^{\infty}L_{k}^{2}}\\
\leq&\left\|(KI_{p+q})(x,k)\right\|_{L_{x}^{\infty}L_{k}^{2}}e^{\left\|U\right\|_{L^{1}}}\\
\lesssim & \left\|U\right\|_{L^{2}}e^{\left\|U\right\|_{L^{1}}}.
\end{aligned}
\end{equation}
Next, we show that $\partial_{k}m^{-}(x,k)\in L_{x}^{\infty}(\mathbb{R}^{-}, L_{k}^{2}(\mathbb{R}))$. It follows from \eqref{2+mo} that
\begin{equation}\label{ppp}
(I-K)(\partial_{k}m^{-})=\partial_{k}(KI_{p+q})+(\partial_{k}K)(m^{-}-I_{p+q}),
\end{equation}
so we only need to prove that
$\partial_{k}(KI_{p+q})$, $(\partial_{k}K)(m^{-}-I_{p+q})\in L_{x}^{\infty}(\mathbb{R}^{-};L_{k}^{2}(\mathbb{R}))$.
Using the fact that $y<x<0$ and a proof similar to \eqref{2+pro1} yields that for $x\in\mathbb{R}^{-}$,
\begin{equation}\label{kki}
\left\|\partial_{k}(KI_{p+q})\right\|_{L_{x}^{\infty}L_{k}^{2}}=\left\|\int_{-\infty}^{x}e^{i(y-x)k\operatorname{ad}\sigma_{3}}i\operatorname{ad}\sigma_{3}\left((y-x)U(y)\right)dy\right\|_{L_{x}^{\infty}L_{k}^{2}}\lesssim \left\|U\right\|_{L^{2,1}}.
\end{equation}
For $(\partial_{k}K)(m^{-}-I_{p+q})$, we first proof that $m^{-}-I_{p+q}\in L_{x}^{2}(\mathbb{R}^{-};L_{k}^{2}(\mathbb{R}))$.
Using the operator equation \eqref{2+mo}, we can obtain that
\begin{align*}
\|m^{-}-I_{p+q}\|_{L_{x}^{2}L_{k}^{2}}\leq \|KI_{p+q}\|_{L_{x}^{2}L_{k}^{2}}+\|K(m^{-}-I_{p+q})\|_{L_{x}^{2}L_{k}^{2}}.
\end{align*}
By using the \eqref{2+pro1p} and  integration by parts, we obtain for every $x\in \mathbb{R}^{-}$,
\begin{align*}
\left\|KI_{p+q}\right\|_{L_{x}^{2}L_{k}^{2}}\lesssim&\left(\int_{-\infty}^{0}\int_{-\infty}^{x}|U(y)|^{2}dydx\right)^{\frac{1}{2}}\\
=&\left(\int_{-\infty}^{0}\left|x\right|\left|U(x)\right|^{2}dx\right)^{\frac{1}{2}}\\
\leq&\left\|U\right\|_{L_{x}^{2,1}}^{\frac{1}{2}}\left\|U\right\|_{L_{x}^{2}}^{\frac{1}{2}}.
\end{align*}
It follows from  the Minkowski inequality that
\begin{align*}
\left\|K(m^{-}-I_{p+q})\right\|_{L_{k}^{2}}=&\left\|\int_{-\infty}^{x}e^{i(y-x)k\operatorname{ad}\sigma_{3}}U(y)\left(m^{-}(y,k)-I_{p+q}\right)dy\right\|_{L_{k}^{2}}\\
\leq&\int_{-\infty}^{x}\left|U(y)\right|dy\left\|m^{-}(x,k)-I_{p+q}\right\|_{L_{x}^{\infty}L_{k}^{2}}.
\end{align*}
By the estimate \eqref{2+m-i} and Hardy's inequality, we have
\begin{align*}
\left\|K(m^{-}-I_{p+q})\right\|_{L_{x}^{2}L_{k}^{2}}\lesssim&\left(\int_{-\infty}^{0}\left(\int_{-\infty}^{x}|U(y)|dy\right)^{2}dx\right)^{\frac{1}{2}}\\
\lesssim&\left(\int_{-\infty}^{0}|U(x)|^{2}x^{2}dx\right)^{\frac{1}{2}}\\
\leq&\left\|U\right\|_{L^{2,1}}.
\end{align*}
Hence, we deduce that $m^{-}-I_{p+q}\in L_{x}^{2}(\mathbb{R}^{-};L_{k}^{2}(\mathbb{R}))$. Then,
we can use the Minkowski inequality, H\"{o}lder inequality and the fact $|x-y|<y$ to estimate
\begin{equation}\label{kkm}
\begin{aligned}
&\left\|(\partial_{k}K)(m^{-}-I_{p+q})\right\|_{L_{k}^{2}}\\
=&\left\|\int_{-\infty}^{x}e^{i(y-x)k\operatorname{ad}\sigma_{3}}i\operatorname{ad}\sigma_{3}\left((y-x)U(y)(m^{-}(y)-I_{p+q})\right)dy\right\|_{L_{k}^{2}}\\
\lesssim&\int_{-\infty}^{x}\left|yU(y)\right|\left\|m^{-}(y,k)-I_{p+q}\right\|_{L_{k}^{2}}dy\\
\leq&\left\|U\right\|_{L^{2,1}}\left\|m^{-}(x,k)-I_{p+q}\right\|_{L_{x}^{2}(\mathbb{R}^{-};L_{k}^{2}(\mathbb{R}))}.
\end{aligned}
\end{equation}
So we have $(\partial_{k}K)(m^{-}-I_{p+q})\in L_{x}^{\infty}(\mathbb{R}^{-};L_{k}^{2}(\mathbb{R}))$.

Next, we proof the Lipschitz continuous of $m^{-}-I_{p+q}$ with respect to $\boldsymbol{Q}$. Suppose that $\boldsymbol{Q}, \tilde{\boldsymbol{Q}} \in H^{1,1}(\mathbb{R})$ satisfy $\|\boldsymbol{Q}\|_{H^{1,1}(\mathbb{R})}, \|\tilde{\boldsymbol{Q}}\|_{H^{1,1}(\mathbb{R})}\leq \gamma$ for some $\gamma>0$. Denote the corresponding Jost functions by $m_{-}$ and $\tilde{m}_{-}$ respectively. From \eqref{2+mk}, we have
\begin{equation*}
\begin{aligned}
m_{-}-\tilde{m}_{-}&=(I-K)^{-1}KI_{p+q}-(I-\tilde{K})^{-1}\tilde{K}I_{p+q}\\
&=(I-K)^{-1}(K-\tilde{K})I_{p+q}+\left((I-K)^{-1}-(I-\tilde{K})^{-1}\right)\tilde{K}I_{p+q}\\
&=(I-K)^{-1}(K-\tilde{K})I_{p+q}+(I-K)^{-1}(K-\tilde{K})(I-\tilde{K})^{-1}\tilde{K}I_{p+q},
\end{aligned}
\end{equation*}
where $\tilde{K}$ can be defined by $\tilde{\boldsymbol{Q}}$ in the same way as  \eqref{2+kk}. Making full use of the definition \eqref{2+kk} as well as the results of Lemma \ref{lk} we have
\begin{equation*}
\left\|(I-K)^{-1}(K-\tilde{K})I_{p+q}\right\|_{L_{x}^{\infty}L_{k}^{2}}\leq c(\gamma)\left\|(K-\tilde{K})I_{p+q}\right\|_{L_{x}^{\infty}L_{k}^{2}}
\leq c(\gamma)\|\boldsymbol{Q}-\tilde{\boldsymbol{Q}}\|_{H^{1,1}}
\end{equation*}
and
\begin{equation}\label{kkk}
\begin{aligned}
&\left\|(I-K)^{-1}(K-\tilde{K})(I-\tilde{K})^{-1}\tilde{K}I_{p+q}\right\|_{L_{x}^{\infty}L_{k}^{2}}\\
\leq& c(\gamma)\left\|(K-\tilde{K})(I-\tilde{K})^{-1}\tilde{K}I_{p+q}\right\|_{L_{x}^{\infty}L_{k}^{2}}\\
\leq& c(\gamma)\|\boldsymbol{Q}-\tilde{\boldsymbol{Q}}\|_{H^{1,1}}\left\|(I-\tilde{K})^{-1}\tilde{K}I_{p+q}\right\|_{L_{x}^{\infty}L_{k}^{2}}\\
\leq& c(\gamma)\|\boldsymbol{Q}-\tilde{\boldsymbol{Q}}\|_{H^{1,1}}.
\end{aligned}
\end{equation}
Then we obtain
\begin{equation}\label{mmm}
\left\|m_{-}-\tilde{m}_{-}\right\|_{L_{x}^{\infty}L_{k}^{2}}
\leq c(\gamma)\|\boldsymbol{Q}-\tilde{\boldsymbol{Q}}\|_{H^{1,1}}.
\end{equation}
Using the equality \eqref{ppp}, we have
\begin{equation}\label{km}
\begin{aligned}
&\partial_{k}m_{-}-\partial_{k}\tilde{m}_{-}\\
=&(I-K)^{-1}\left(\partial_{k}(KI_{p+q}-\tilde{K}I_{p+q})\right)+(I-K)^{-1}\left((\partial_{k}K)(m^{-}-\tilde{m}^{-})\right)\\
&+(I-K)^{-1}\left((\partial_{k}K-\partial_{k}\tilde{K})(\tilde{m}^{-}-I_{p+q})\right)\\
&-(I-K)^{-1}(K-\tilde{K})(I-\tilde{K})^{-1}\left(\partial_{k}(\tilde{K}I_{p+q})+(\partial_{k}\tilde{K})(\tilde{m}^{-}-I_{p+q})\right).
\end{aligned}\nonumber
\end{equation}
From \eqref{kki}, we obtain
\begin{equation*}
\|\partial_{k}(KI_{p+q}-\tilde{K}I_{p+q})\|_{L_{x}^{\infty}L_{k}^{2}}\leq c(\gamma)\|\boldsymbol{Q}-\tilde{\boldsymbol{Q}}\|_{H^{1,1}}.
\end{equation*}
By using \eqref{kki}-\eqref{mmm} and repeating the same analysis to \eqref{km}, we have
\begin{equation*}
\left\|\partial_{k}m_{-}-\partial_{k}\tilde{m}_{-}\right\|_{L_{x}^{\infty}L_{k}^{2}}
\leq c(\gamma)\|\boldsymbol{Q}-\tilde{\boldsymbol{Q}}\|_{H^{1,1}}.
\end{equation*}
This completes the proof of Proposition \ref{lm} for $m^{-}(x,k)$.
\end{proof}

\subsection{Regularity of the matrix scattering data}\label{s22}
By the uniqueness theory of ODE and $\det m^{\pm}=1$, we can define the scattering matrix associated with the spectral problem \eqref{2.3} as follows
\begin{equation}\label{2+abcd}
m^{-}(x,k)=m^{+}(x,k)e^{-ixk\operatorname{ad}\sigma_{3}}S(k),\quad S(k)=\begin{bmatrix}\boldsymbol{A}(k)&\boldsymbol{B}(k)\\\boldsymbol{C}(k)&\boldsymbol{D}(k)\end{bmatrix},\quad k\in\mathbb{R},
\end{equation}
where scattering date $\boldsymbol{A}(k), \boldsymbol{B}(k), \boldsymbol{C}(k)$ and $\boldsymbol{D}(k)$ are the $p\times p$, $p\times q$, $q\times p$ and $q\times q$ matrices, respectively. It follows from the analyticity of the eigenfunctions $m^{\pm}$ that $\boldsymbol{A}(k)$ and $\boldsymbol{D}(k)$ can be analytically extended to the upper and the lower half-planes, respectively. Clearly, $\det S(k)=1$. And $S(k)$ can be expressed in the integral form
\begin{align}\label{2+s}
S(k)=I_{p+q}+\int_{\mathbb{R}}e^{iyk\operatorname{ad}\sigma_{3}}U(y)m^{-}(y,k)dy.
\end{align}
Starting from the integral expression \eqref{2+s} and making full use of the results in Proposition \ref{lm}, we can obtain the regularity of the scattering data as follows.
\begin{proposition}\label{lmab}
If $\boldsymbol{Q}\in H^{1,1}(\mathbb{R})$, then
$\boldsymbol{A}-I_{p}, \boldsymbol{D}-I_{q}\in H^{1}(\mathbb{R})$ and $\boldsymbol{B}, \boldsymbol{C}\in H^{1,1}(\mathbb{R})$. Moreover, the map
\begin{equation}\label{qa}
\boldsymbol{Q}\mapsto \left[\boldsymbol{A}-I_{p}, \boldsymbol{D}-I_{q}, \boldsymbol{B}, \boldsymbol{C}\right]
\end{equation}
is Lipschitz continuous.
\end{proposition}
\begin{proof}
The idea of this proof was first given by Zhou \cite{Zhou1998}, and we still give a detailed proof progress to show that the conclusions are still
hold for the scattering data in matrix form.
The integral equation \eqref{2+s} can be rewritten in the form
\begin{equation}\label{2+s1}
\begin{aligned}
S(k)&=I_{p+q}+\int_{\mathbb{R}}e^{iyk\operatorname{ad}\sigma_{3}}U(y)\left(m^{-}(y,k)-I_{p+q}\right)dy+\int_{\mathbb{R}}e^{iyk\operatorname{ad}\sigma_{3}}U(y)dy,\\
&\equiv I_{p+q}+S_{1}(k)+S_{2}(k).
\end{aligned}
\end{equation}
Clearly,
$S_{2}(k)\in H^{1,1}(\mathbb{R})$ by Fourier transform, so we only need to focus on $S_{1}(k)$.

Let us first prove that $S_{1}(k)\in H^{1}(\mathbb{R})$. By \eqref{2+m-i}, we can use the Minkowski inequality to estimate
\begin{align*}
\left\|S_{1}\right\|_{L_{k}^{2}}&=\left\|\int_{\mathbb{R}}e^{iyk\operatorname{ad}\sigma_{3}}U(y)\left(m^{-}(y,k)-I_{p+q}\right)dy\right\|_{L_{k}^{2}}\\
&\leq\left\|m^{-}(x,k)-I_{p+q}\right\|_{L_{x}^{\infty}L_{k}^{2}}\left\|U\right\|_{L^{1}}\\
&\lesssim \left\|U\right\|_{L^{2}}\left\|U\right\|_{L^{1}}e^{\left\|U\right\|_{L^{1}}}.
\end{align*}
Considering the equation \eqref{2+abcd} at $x=0$, we have
\begin{align*}
\partial_{k}S(k)=\partial_{k}(m^{+}(0,k)^{-1})m^{-}(0,k)+m^{+}(0,k)^{-1}\partial_{k}m^{-}(0,k),
\end{align*}
then $\partial_{k}S(k)\in L^{2}(\mathbb{R})$ follows from Proposition \ref{lm}.

Next, we analyze the $kS_{1}(k)$ in order to show that $\boldsymbol{B}(k), \boldsymbol{C}(k) \in L^{2,1}(\mathbb{R})$. Notice that the function $(KI_{p+q})(x,k)$ can be rewritten through integration by parts as follows
\begin{equation*}
\begin{aligned}
(KI_{p+q})(x,k)&=\int_{-\infty}^{x}e^{i(y-x)k\operatorname{ad}\sigma_{3}}U(y)dy\\
&=\frac{1}{ik}(\operatorname{ad}\sigma_{3})^{-1}U(x)-\frac{1}{ik}(\operatorname{ad}\sigma_{3})^{-1}\int_{-\infty}^{x}e^{i(y-x)k\operatorname{ad}\sigma_{3}}U'(y)dy\\
&\triangleq  h_{1}(x,k)+h_{2}(x,k),
\end{aligned}
\end{equation*}
then from \eqref{2+mk} we have
\begin{equation*}
kS_{1}(k)=\int_{\mathbb{R}}e^{iyk\operatorname{ad}\sigma_{3}}U(y)k(I-K)^{-1}\left(h_{1}(y,k)+h_{2}(y,k)\right)dy.
\end{equation*}
For $h_{2}(x,k)$, the application of Lemma \ref{lk} and Minkowski inequality gives
\begin{align*}
&\left\|\int_{\mathbb{R}}e^{iyk\operatorname{ad}\sigma_{3}}U(y)k(I-K)^{-1}h_{2}(y,k)dy\right\|_{L_{k}^{2}}\\
\leq&\left\|(I-K)^{-1}kh_{2}(x,k)\right\|_{L_{x}^{\infty}L_{k}^{2}}\left\|U\right\|_{L^{1}}\\
\leq&\left\|kh_{2}(x,k)\right\|_{L_{x}^{\infty}L_{k}^{2}}e^{\left\|U\right\|_{L^{1}}}\left\|U\right\|_{L^{1}}\\
\lesssim & \|U'\|_{L^{2}}\left\|U\right\|_{L^{1}}e^{\left\|U\right\|_{L^{1}}}.
\end{align*}
The analysis of $h_{1}(x,k)$ will be more complicated, so we first decompose $(I-K)^{-1}h_{1}$ into the following form
\begin{align*}
(I-K)^{-1}h_{1}(x,k)&=h_{1}(x,k)+Kh_{1}(x,k)+(I-K)^{-1}K^{2}h_{1}(x,k)\\
&\triangleq h_{1}(x,k)+g_{1}(x,k)+g_{2}(x,k).
\end{align*}
Our goal is to prove that
\begin{align*}
\int_{\mathbb{R}}e^{iyk\operatorname{ad}\sigma_{3}}U(y)k\left(h_{1}(y,k)+g_{1}(y,k)+g_{2}(y,k)\right)dy\in L^{2}(\mathbb{R}).
\end{align*}
We notice that $\int_{\mathbb{R}}e^{iyk\operatorname{ad}\sigma_{3}}U(y)kh_{1}(y,k)dy$ does no contribute to $\boldsymbol{B}(k)$ and $\boldsymbol{C}(k)$ since it is a block diagonal matrix.
For $kg_{1}(x,k)$, we have
\begin{align*}
&\left\|\int_{\mathbb{R}}e^{iyk\operatorname{ad}\sigma_{3}}U(y)kg_{1}(y,k)dy\right\|_{L_{k}^{2}}\\
=&\left\|\int_{\mathbb{R}}e^{iyk\operatorname{ad}\sigma_{3}}U(y)\left(\int_{-\infty}^{y}U(t)(i\operatorname{ad}\sigma_{3})^{-1}U(t)dt\right)dy\right\|_{L_{k}^{2}}\\
\lesssim&\left\|U\right\|_{L^{2}}^{2},
\end{align*}
the last inequality can be obtained by triangularity and Fourier transform.
For $kg_{2}(x,k)$, by using Minkowski inequality, Lemma \ref{lk} and an analogous method of proof with $kg_{1}$, we obtain
\begin{align*}
&\left\|\int_{\mathbb{R}}e^{iyk\operatorname{ad}\sigma_{3}}U(y)(I-K)^{-1}K^{2}kh_{1}(y,k)dy\right\|_{L_{k}^{2}}\\
\leq&\left\|(I-K)^{-1}K^{2}kh_{1}\right\|_{L_{x}^{\infty}L_{k}^{2}}\left\|U\right\|_{L^{1}}
\leq\left\|K^{2}kh_{1}\right\|_{L_{x}^{\infty}L_{k}^{2}}\left\|U\right\|_{L^{1}}e^{\left\|U\right\|_{L^{1}}}\\
=&\left\|\int_{-\infty}^{x}e^{i(y-x)k\operatorname{ad}\sigma_{3}}U(y)\left(\int_{-\infty}^{y}U(t)(i\operatorname{ad}\sigma_{3})^{-1}U(t)dt\right)dy\right\|_{L_{x}^{\infty}L_{k}^{2}}\left\|U\right\|_{L^{1}}e^{\left\|U\right\|_{L^{1}}}\\
\lesssim&\left\|U\right\|_{L^{1}}^{2}\left\|U\right\|_{L^{1}}e^{\left\|U\right\|_{L^{1}}}.
\end{align*}
So we get the smoothness and decay properties of scattering date in matrix form. Suppose that $\boldsymbol{Q}, \ \tilde{\boldsymbol{Q}}\in H^{1,1}(\mathbb{R})$ satisfy $\|\boldsymbol{Q}\|_{H^{1,1}(\mathbb{R})}, \|\tilde{\boldsymbol{Q}}\|_{H^{1,1}(\mathbb{R})}<\gamma$ for some $\gamma>0$ and denote the corresponding scattering matrix by $S(k)$ and $\tilde{S}(k)$ respectively. Using the Lipschitz continuous in Proposition \ref{lm} and repeating the above proof we can obtain the Lipschitz continuous for scattering date. Therefore, we complete the proof.
\end{proof}

We next introduce the matrix reflection coefficient, which plays a crucial role in the inverse scattering transform method. For the scattering data in the form of matrices, we can define the relevant $p\times q$ matrix reflection coefficient as
\begin{equation}\label{2+r}
\boldsymbol{R}(k):=\boldsymbol{B}(k)\boldsymbol{D}^{-1}(k),\quad k\in\mathbb{R},
\end{equation}
where $\boldsymbol{R}(k)$ is well-defined when $\det \boldsymbol{D}(k)\neq 0$. The regularity, decay and the Lipschitz continuous associated with the reflection coefficient $\boldsymbol{R}(k)$ can be can be summarized by the following Proposition, which can be deduced directly from the Proposition \ref{lmab}.
\begin{proposition}\label{pro1}
If $\boldsymbol{Q}\in H^{1,1}(\mathbb{R})$, then $\boldsymbol{R}\in H^{1,1}(\mathbb{R})$ and the map $\boldsymbol{Q}\mapsto \boldsymbol{R}$ is Lipschitz continuous.
\end{proposition}
A very natural question is how to guarantee the invertibility of $\boldsymbol{D}(k)$ and build the unique solvable RH problem. The situations of the defocusing and focusing matrix NLS equations are different, so we consider them separately in the following Section \ref{s3} and Section \ref{s4}.

\section{Inverse Scattering Map for Defocusing Case}\label{s3}

In this section, we  study the  inverse scattering  transform  and further
prove the  well-posedness  in space $H^{1,1}(\mathbb{R})$  for the
 defocusing $p\times q$ matrix NLS equation.

\subsection{Symmetry reduction}\label{s31}

By rewriting the equation \eqref{2.1} as
\begin{equation*}
(i\sigma_{3}\partial_{x}-i\sigma_{3}U)\psi=k\psi,
\end{equation*}
it can be obtained that in this case the scattering problem is self-adjoint, and therefore the corresponding eigenvalues need to be real. This is consistent with the scalar defocusing NLS equation. Further, this implies that $\det \boldsymbol{D}(k)$ and  $\det \boldsymbol{A}(k)$ have no nonreal zeros.

By observing the features of the matrix $U$ defined in \eqref{Q}, it is easy to obtain the symmetry relation
\begin{equation}\label{Qd}
U^{\dag}= U.
\end{equation}
With the help of \eqref{2.3} and \eqref{Qd}, we can give the symmetry relation
\begin{equation}
(m^{\pm})^{-1}(k)=\sigma_{3}(m^{\pm})^{\dag}(k^{*})\sigma_{3}
\end{equation}
for the eigenfunctions. And then, the symmetry of the scattering matrix
\begin{equation}\label{2+s-1}
S^{-1}(k)=\sigma_{3}S^{\dag}(k^{*})\sigma_{3}
\end{equation}
follows from \eqref{2+abcd}. For $k\in\mathbb{R}$, it can be further obtained from \eqref{2+s-1} that
\begin{align*}
&|\det \boldsymbol{D}|^{2}=\det\left(\boldsymbol{D}\boldsymbol{D}^{\dag}\right)=\det\left(I_{q}+\boldsymbol{C}\boldsymbol{C}^{\dag}\right)>0,\\
&|\det \boldsymbol{A}|^{2}=\det\left(\boldsymbol{A}\boldsymbol{A}^{\dag}\right)=\det\left(I_{p}+\boldsymbol{B}\boldsymbol{B}^{\dag}\right)>0.
\end{align*}
Thus $\det \boldsymbol{D}(k)$ and $\det \boldsymbol{A}(k)$  have no poles or spectral the scattering data $\boldsymbol{D}(k)$ and $\boldsymbol{A}(k)$ are naturally invertible for $k\in\mathbb{C}^{-}\cup\mathbb{R}$ and $k\in\mathbb{C}^{+}\cup\mathbb{R}$, $k\in\mathbb{C}^{+}\cup\mathbb{R}$, respectively. Making use of symmetry relation \eqref{2+s-1} we have
\begin{equation}\label{ca1}
\boldsymbol{C}(k)\boldsymbol{A}^{-1}(k)=\left(\boldsymbol{B}(k)\boldsymbol{D}^{-1}(k)\right)^{\dag}= \boldsymbol{R}^{\dag}(k).
\end{equation}

\subsection{Set-up of an  RH problem }\label{s32}
We construct a piecewise analytic matrix function from \eqref{2+abcd}, which will be an essential part of the subsequent RH problem. Let
\begin{equation}\label{mxk1}
M(x,k)=\left\{
\begin{aligned}
&\begin{bmatrix}m_{1}^{-}(x,k)\boldsymbol{A}^{-1}(k),&m_{2}^{+}(x,k)\end{bmatrix},\quad \text{Im}k>0,\\
&\begin{bmatrix}m_{1}^{+}(x,k),&m_{2}^{-}(x,k)\boldsymbol{D}^{-1}(k)\end{bmatrix},\quad \text{Im}k<0,
\end{aligned}
\right.
\end{equation}
we have
\begin{equation}\label{M1}
M_{+}(x,k)=M_{-}(x,k)e^{-ixk\operatorname{ad}\sigma_{3}}v(k),\quad k\in\mathbb{R},
\end{equation}
where
\begin{equation}\label{v1}
v(k)=\begin{bmatrix}I_{p}- \boldsymbol{R}(k)\boldsymbol{R}^{\dag}(k^{*})&-\boldsymbol{R}(k)\\
 \boldsymbol{R}^{\dag}(k^{*})&I_{q}\end{bmatrix},
\end{equation}
and
\begin{equation}\label{vpm}
      v(k)=v_{-}^{-1}(k)v_{+}(k)=\begin{bmatrix}I_{p}&-\boldsymbol{R}(k)\\0_{q\times p}&I_{q}\end{bmatrix}
      \begin{bmatrix}I_{p}&0_{p\times q}\\ \boldsymbol{R}^{\dag}(k)&I_{q}\end{bmatrix}.
\end{equation}
For defocusing ($\sigma=1$) matrix NLS equation, neither $\det \boldsymbol{A}(k)$ nor $\det \boldsymbol{D}(k)$ is equal to zero, so $M(x,k)$ is always well-defined.

It follows from \eqref{2.5} that
\begin{equation}\label{m1m2}
\left[m_{1}^{-}, m_{2}^{+}\right]=I_{p+q}+\int_{\delta}^{x}e^{i(y-x)k\operatorname{ad}\sigma_{3}}U(y)\left[m_{1}^{-}(y), m_{2}^{+}(y)\right]dy,
\end{equation}
where $\delta=-\infty$ and $\delta=+\infty$ for the first and second columns of the $2\times2$ block matrix, respectively. Combining \eqref{2.5} and \eqref{2+s}, we have
\begin{equation}\label{m11a}
m_{11}^{-}(x,k)=\boldsymbol{A}(k)+\int_{+\infty}^{x}\boldsymbol{Q}(y)m_{21}^{-}(y,k)dy.
\end{equation}
By substituting \eqref{m11a} into \eqref{m1m2}, we obtain
\begin{equation}\label{m1a}
\left[m_{1}^{-}\boldsymbol{A}^{-1}, m_{2}^{+}\right]=I_{p+q}+\int_{\delta}^{x}e^{i(y-x)k\operatorname{ad}\sigma_{3}}U(y)\left[m_{1}^{-}(y)\boldsymbol{A}^{-1}, m_{2}^{+}(y)\right]dy,
\end{equation}
where $\delta=-\infty$ for the (2,1) entry and $\delta=+\infty$ for the (1,1), (1,2) and (2,2) entries of the $2\times2$ block matrix. And similarly,
\begin{equation}\label{m2d}
\left[m_{1}^{+}, m_{2}^{-}\boldsymbol{D}^{-1}\right]=I_{p+q}+\int_{\delta}^{x}e^{i(y-x)k\operatorname{ad}\sigma_{3}}U(y)\left[m_{1}^{+}(y), m_{2}^{-}(y)\boldsymbol{D}^{-1}\right]dy,
\end{equation}
where $\delta=-\infty$ for the (1,2) entry and $\delta=+\infty$ for the (1,1), (2,1) and (2,2) entries of the $2\times2$ block matrix. The integral equations \eqref{m1a} and \eqref{m2d} imply that $M(x,k)$ satisfies the differential equation \eqref{2.3}. Moreover, we have $M(x,k)\rightarrow I_{p+q} \ \text{as} \ x\rightarrow+\infty$ according to \eqref{mi}, \eqref{2+abcd} and an application of the Lebesgue's dominated convergence theorem for $m_{21}^{-}(x,k)$ and $m_{12}^{-}(x,k)$. Hence, $M(x,k)$ defined by \eqref{mxk1} is right-normalized. In fact, we can also construct a left-normalized piecewise analytic matrix function $\widetilde{M}(x,k)$ as follows
\begin{equation}\label{ml}
\widetilde{M}(x,k)=\left\{
\begin{aligned}
&\begin{bmatrix}m_{1}^{-}(x,k),&m_{2}^{+}(x,k)(\boldsymbol{D}^{-1})^{\dag}(k^{*})\end{bmatrix},\quad \text{Im}k>0,\\
&\begin{bmatrix}m_{1}^{+}(x,k)(\boldsymbol{A}^{-1})^{\dag}(k^{*}),&m_{2}^{-}(x,k)\end{bmatrix},\quad \text{Im}k<0,
\end{aligned}
\right.
\end{equation}
which satisfies $\widetilde{M}(x,k)\rightarrow I_{p+q} \ \text{as} \ x\rightarrow -\infty$. Let us define a auxiliary matrix $\delta(k)$ by
\begin{equation}\label{del}
\delta(x)=\left\{
\begin{aligned}
&\begin{bmatrix}\boldsymbol{A}(k)&0_{p\times q}\\0_{q\times p}&(\boldsymbol{D}^{-1})^{\dag}(k^{*})\end{bmatrix},\quad \text{Im}k>0,\\
&\begin{bmatrix}(\boldsymbol{A}^{-1})^{\dag}(k^{*})&0_{p\times q}\\0_{q\times p}&\boldsymbol{D}(k)\end{bmatrix},\quad \text{Im}k<0.
\end{aligned}
\right.
\end{equation}
The relationship between $\widetilde{M}(x,k)$ and $M(x,k)$ can be established by $\delta(k)$, i.e.
\begin{equation}
\widetilde{M}(x,k)=M(x,k)\delta(k).
\end{equation}
Moreover, the jump matrix $\tilde{v}(k)$ for $\widetilde{M}(x,k)$ can be given by
\begin{equation}
\tilde{v}(k)=\delta_{-}^{-1}(k)v(k)\delta_{+}(k),
\end{equation}
where $v(k)$ is defined by \eqref{v1}.

\subsection{Time evolution of the matrix scattering data}\label{s33}
Given the fundamental solutions $M_{\pm}(x,t,k)$ of the Lax pair
\begin{align}
m_{x}(x,t,k)&=-ik\operatorname{ad}\sigma_{3}m(x,t,k)+U(x,t)m(x,t,k),\label{x1}\\
m_{t}(x,t,k)&=-2ik^{2}\operatorname{ad}\sigma_{3}m(x,t,k)+P(x,t,k)m(x,t,k),\label{t1}
\end{align}
satisfying
\begin{equation}
M_{+}(x,t,k)=M_{-}(x,t,k)e^{-ixk\operatorname{ad}\sigma_{3}}v(t,k).
\end{equation}
Here,
$P(x,t,k)=2kU(x,t)+i\sigma_{3}U_{x}(x,t)-i\sigma_{3}U^{2}(x,t)$.
Substituting
\begin{equation*}
M_{-}(x,t,k)e^{-ixk\sigma_{3}}\ \text{and} \ M_{-}(x,t,k)e^{-ixk\sigma_{3}}v(t,k)
\end{equation*}
into \eqref{t1}, and using the fact that $M_{\pm}(x,t,k)\rightarrow I_{p+q}$ as $x\rightarrow+\infty$, we obtain
\begin{equation*}
v_{t}(t,k)=-2ik^{2}\operatorname{ad}\sigma_{3}v(t,k),
\end{equation*}
so we have
\begin{equation}
v(t,k)=e^{-2itk^{2}\operatorname{ad}\sigma_{3}}v(0,k),
\end{equation}
and
\begin{equation}\label{vt1}
v_{\pm}(t,k)=e^{-2itk^{2}\operatorname{ad}\sigma_{3}}v_{\pm}(0,k).
\end{equation}
Then we can extend the results of Proposition \ref{pro1} to $v_{\pm}(t,k)$.
\begin{proposition}\label{ct1}
If $\boldsymbol{Q}_{0}\in H^{1,1}(\mathbb{R})$, then for every $t\in[-T,T],\ T>0$, we have $v_{\pm}(t,k)-I_{p+q}\in H_{k}^{1,1}(\mathbb{R})$. Moreover, the mapping $\boldsymbol{Q}_{0}\mapsto v_{\pm}-I_{p+q}$ is Lipschitz continuous from $H^{1,1}(\mathbb{R})$ to $C([-T,T], H^{1,1}(\mathbb{R}))$ for every $T>0$.
\end{proposition}
\begin{proof}
For every $t\in[-T,T]$, we have
\begin{equation*}
\|v_{\pm}(t,k)-I_{p+q}\|_{L^{2,1}}\lesssim\|v_{\pm}(0,k)-I_{p+q}\|_{L^{2,1}}, \end{equation*}
\begin{equation*}
\begin{aligned}
\left\|\partial_{k}v_{\pm}(t,k)\right\|_{L^{2}}
&\lesssim\left\|-4itk\operatorname{ad}\sigma_{3}
v_{\pm}(0,k)+\partial_{k}v_{\pm}(0,k)\right\|_{L^{2}}\\
&\lesssim4T\left\|v_{\pm}(0,k)\right\|_{L^{2,1}}+\left\|\partial_{k}v_{\pm}(0,k)\right\|_{L^{2}},
\end{aligned}
\end{equation*}
thus
\begin{equation}\label{v12}
\|v_{\pm}(t,k)-I_{p+q}\|_{H^{1,1}}\lesssim\|v_{\pm}(0,k)-I_{p+q}\|_{H^{1,1}}.
\end{equation}
Next, we show that $v(t,k)$ is uniformly continuous with respect to time $t$ in the sense of the norm $H^{1,1}(\mathbb{R})$. An application of the Lebesgue's dominated convergence theorem implies that for any $\varepsilon>0$, there exists $\gamma>0$ such that
\begin{equation*}
\|v_{\pm}(0,k)-I_{p+q}\|_{H^{1,1}(|k|>\gamma)}<\varepsilon.
\end{equation*}
Then, for every $t_{1}, t_{2}\in [-T,T]$, we have
\begin{align*}
\|v_{\pm}(t_{1},k)-v_{\pm}(t_{2},k)\|_{H^{1,1}(|k|>\gamma)}\lesssim\|v_{\pm}(0,k)-I_{p+q}\|_{H^{1,1}(|k|>\gamma)}<\varepsilon.
\end{align*}
When $|k|\leq\gamma$, we can obtain the estimates
\begin{align}
&\left|e^{\pm2i(t_{1}-t_{2})k^{2}\sigma_{3}}-I_{p+q}\right|=\left|2ik^{2}\sigma_{3}\int_{0}^{t_{1}-t_{2}}e^{\pm2isk^{2}\sigma_{3}}ds\right|\lesssim\gamma^{2}\left|t_{1}-t_{2}\right|,\label{e11}\\
&\left|e^{\pm2it_{1}k^{2}\sigma_{3}}t_{1}-e^{\pm2it_{2}k^{2}\sigma_{3}}t_{2}\right|\lesssim(\gamma^{2}T+1)\left|t_{1}-t_{2}\right|.\label{e21}
\end{align}
By the relation \eqref{vt1} and \eqref{e11}, we have
\begin{align*}
\left\|v_{\pm}(t_{1},k)-v_{\pm}(t_{2},k)\right\|_{L^{2,1}(|k|\leq\gamma)}\lesssim&\left\|v_{\pm}(0,k)\left(e^{2i(t_{1}-t_{2})k^{2}\sigma_{3}}-I_{p+q}\right)\right\|_{L^{2,1}(|k|\leq\gamma)}
\\&+\left\|\left(e^{-2i(t_{1}-t_{2})k^{2}\sigma_{3}}-I_{p+q}\right)v_{\pm}(0,k)\right\|_{L^{2,1}(|k|\leq\gamma)}\\
\lesssim&\gamma^{2}\left|t_{1}-t_{2}\right|\left\|v_{\pm}(0,k)\right\|_{L^{2,1}(|k|\leq\gamma)}.
\end{align*}
Make use of the estimates \eqref{e11} and \eqref{e21}, we obtain
\begin{align*}
&\left\|\partial_{k}(v_{\pm}(t_{1},k)-v_{\pm}(t_{2},k))\right\|_{L^{2}(|k|\leq\gamma)}\\
\lesssim&\left\|t_{1}k\operatorname{ad}\sigma_{3}v_{\pm}(0,k)\left(e^{2it_{1}k^{2}\sigma_{3}}-e^{2it_{2}k^{2}\sigma_{3}
}\right)\right\|_{L^{2}(|k|\leq\gamma)}\\
&+\left\|\left(e^{-2it_{1}k^{2}\sigma_{3}}t_{1}-e^{-2it_{2}k^{2}\sigma_{3}}t_{2}\right)k\operatorname{ad}\sigma_{3}v_{\pm}(0,k)\right\|_{L^{2}(|k|\leq\gamma)}\\
&+\left\|\left(e^{-2it_{1}k^{2}\operatorname{ad}\sigma_{3}}-e^{-2it_{2}k^{2}\operatorname{ad}\sigma_{3}}\right)\partial_{k}v_{\pm}(0,k)\right\|_{L^{2}(|k|\leq\gamma)}\\
\lesssim&\left(2\gamma^{2}T+1\right)\left|t_{1}-t_{2}\right|\left\|v_{\pm}(0,k)\right\|_{L^{2,1}(|k|\leq\gamma)}+\gamma^{2}\left|t_{1}-t_{2}\right|\left\|\partial_{k}v_{\pm}(0,k)\right\|_{L^{2}(|k|\leq\gamma)}.
\end{align*}
Therefore, in the sense of the norm $H^{1,1}(\mathbb{R})$, $v(t,k)$ is uniformly continuous with respect to time $t$. The Lipschitz continuity follows from Proposition \ref{pro1}.
\end{proof}

\subsection{Unique solvability of the RH problem}\label{s34}
For the case $\sigma=1$: defocusing matrix NLS equation, let us consider the following RH problem:
\begin{rhp}\label{RH1}
Find a $(p+q)\times(p+q)$ matrix-valued function $M(x,k)$ that satisfies the following conditions:
\begin{itemize}
  \item Analyticity: $M(x,k)$ is analytic for $k\in\mathbb{C}\setminus\mathbb{R}$.
  \item Jump condition:  $M$ has continuous boundary values $M_{\pm}(x,k)$ as $k'\rightarrow k$ from $\mathbb{C}^{\pm}$ to $\mathbb{R}$ and
      \begin{equation}\label{vx}
      M_{+}(x,k)=M_{-}(x,k)v_{x}(k),\quad k\in\mathbb{R},
      \end{equation}
      where
      \begin{equation}\label{vk}
       v_{x}(k)=e^{-ixk\operatorname{ad \sigma_{3}}}v(k)=e^{-ixk\operatorname{ad \sigma_{3}}}\begin{bmatrix}I_{p}- \boldsymbol{R}(k)\boldsymbol{R}^{\dag}(k)&-\boldsymbol{R}(k)\\ \boldsymbol{R}^{\dag}(k)&I_{q}\end{bmatrix}.
      \end{equation}
  \item Normalization: $M(x,k)=I_{p+q}+\mathcal{O}(k^{-1})$ as $k\rightarrow\infty$.
\end{itemize}
\end{rhp}
In order to prove the Lipschitz continuity of the inverse scattering map, below we consider the solvability of the RH problems in a larger space.
\begin{proposition}\label{prh1}
If $v_{\pm}(k)-I_{p+q}\in H^{\frac{1}{2}+\varepsilon}(\mathbb{R})$, $\varepsilon>0$, then there exists a unique solution to the RH problem \ref{RH1} for every $x\in\mathbb{R}$.
\end{proposition}
\begin{proof}
It is easy to see that $v_{x}(k)$ admits the triangular factorization
\begin{align}\label{vv}
v_{x}(k)=v_{x-}^{-1}(k)v_{x+}(k)=\begin{bmatrix}I_{p}&-\boldsymbol{R}(k)e^{-2ixk}\\0_{q\times p}&I_{q}\end{bmatrix}
\begin{bmatrix}I_{p}&0_{p\times q}\\ \boldsymbol{R}^{\dag}(k)e^{2ixk}&I_{q}\end{bmatrix}.
\end{align}
We define two nilpotent matrices $w_{x}^{+}$ and $w_{x}^{-}$ by
\begin{equation}
w_{x}^{+}:=v_{x+}-I_{p+q}, \quad w_{x}^{-}:=I_{p+q}-v_{x-}.
\end{equation}
According to the classical RH theory, it is sufficient for this proposition to prove that the following Beals-Coifman integral
equation can be solved uniquely
\begin{equation}\label{nu}
\nu(x,k)=I_{p+q}+(\mathcal{C}_{w_{x}}\nu)(x,k),
\end{equation}
where
\begin{equation}\label{cw}
\mathcal{C}_{w_{x}}\nu=\mathcal{C}_{\mathbb{R}}^{+}(\nu w_{x}^{-})+\mathcal{C}_{\mathbb{R}}^{-}(\nu w_{x}^{+}),
\end{equation}
and
\begin{equation}\label{bbb}
\nu=M_{+}(I_{p+q}+w_{x}^{+})^{-1}=M_{-}(I_{p+q}-w_{x}^{-})^{-1}.
\end{equation}
The matrices $w_{x}^{+}$ and $w_{x}^{-}$ allow the uniform rational approximations
since $H^{\frac{1}{2}+\varepsilon}(\mathbb{R})\subset C_{0}(\mathbb{R})$, for every $\varepsilon>0$.
Furthermore, the operator $(I-\mathcal{C}_{w_{x}})$ is a Fredholm operator and has zero Fredholm index follow from the Propositions 4.1 and 4.2 in \cite{Zhou1989b}.

From \eqref{vk}, we have
\begin{equation}
v(k)+v^{\dag}(k)=2\begin{bmatrix}I_{p}-\boldsymbol{R}(k)\boldsymbol{R}^{\dag}(k)&0_{p\times q}\\ 0_{q\times p}&I_{q}\end{bmatrix}.
\end{equation}
By \eqref{2+s-1} and \eqref{ca1}, we obtain
\begin{equation*}
\begin{aligned}
I_{p}- \boldsymbol{R}\boldsymbol{R}^{\dag}&=I_{p}-\boldsymbol{B}\boldsymbol{D}^{-1}\boldsymbol{C}\boldsymbol{A}^{-1}=(\boldsymbol{A}-\boldsymbol{B}\boldsymbol{D}^{-1}\boldsymbol{C})\boldsymbol{A}^{-1}\\
&=(\boldsymbol{A}-\boldsymbol{B}\boldsymbol{B}^{\dag}(\boldsymbol{A}^{-1})^{\dag})\boldsymbol{A}^{-1}=(\boldsymbol{A}\boldsymbol{A}^{\dag}-\boldsymbol{B}\boldsymbol{B}^{\dag})(\boldsymbol{A}^{-1})^{\dag}\boldsymbol{A}^{-1}\\
&=(\boldsymbol{A}^{-1})^{\dag}\boldsymbol{A}^{-1},
\end{aligned}
\end{equation*}
then $v(k)+v^{\dag}(k)$ can be decomposed as
\begin{equation}
v(k)+v^{\dag}(k)=2\begin{bmatrix}(\boldsymbol{A}^{-1})^{\dag}&0_{p\times q}\\ 0_{q\times p}&I_{q}\end{bmatrix}\begin{bmatrix}\boldsymbol{A}^{-1}&0_{p\times q}\\ 0_{q\times p}&I_{q}\end{bmatrix}.
\end{equation}
Thus $v(k)+v^{\dag}(k)$ is positive definite on $\mathbb{R}$. It follows from Theorem 9.3 in \cite{Zhou1989b} that $\ker_{L^{2}(\mathbb{R})}(I-\mathcal{C}_{w_{x}})=0$, then $(I-\mathcal{C}_{w_{x}})$ is invertible in $L^{2}(\mathbb{R})$ by  Fredholm theory. Moreover, the solution to the RH problem \ref{RH1} is given by
\begin{equation}\label{mm}
\begin{aligned}
M(x,k)&=I_{p+q}+\mathcal{C}_{\mathbb{R}}\left(\nu(w_{x}^{+}+w_{x}^{-})\right)(x,k)\\
&=I_{p+q}+\frac{1}{2\pi i}\int_{\mathbb{R}}\frac{\nu(x,s)\left(w_{x}^{+}(s)+w_{x}^{-}(s)\right)}{s-k}ds.
\end{aligned}
\end{equation}
\end{proof}
\subsection{Estimates of the reconstructed potential}\label{s35}
According to the fact that $M(x,k)$ satisfies \eqref{2.3} and using the asymptotic expansion of $M(x,k)$ it follows that the potential $\boldsymbol{Q}(x)$ can be recovered by the following formula
\begin{equation}\label{re0}
\boldsymbol{Q}(x)=\lim_{k\rightarrow\infty}2ikM_{12}(x,k),
\end{equation}
where the ``12" subscript denotes the (1,2) entry of the $2\times2$ block matrix. From \eqref{mm}, the above reconstruction formula can be rewritten as
\begin{equation}\label{r1}
\boldsymbol{Q}(x)=\left(-\frac{1}{\pi}\int_{\mathbb{R}}\nu(x,k)\left(v_{x+}(k)-v_{x-}(k)\right)dk\right)_{12}.
\end{equation}

We rewrite the Beals-Coifman integral operator $\mathcal{C}_{w_{x}}$ in \eqref{cw} associated with \eqref{r1} in the equivalent form
\begin{equation}\label{cw1}
\mathcal{C}_{v_{x}}\phi:=\mathcal{C}_{\mathbb{R}}^{+}\left(\phi\left(I_{p+q}-v_{x-}\right)\right)+\mathcal{C}_{\mathbb{R}}^{-}\left(\phi(v_{x+}-I_{p+q})\right).
\end{equation}
Denote $v\in Y$ if $v_{\pm}(k)-I_{p+q}\in H^{\frac{1}{2}+\varepsilon}(\mathbb{R})$ and $v\in Z$ if $v_{\pm}(k)-I_{p+q}\in H^{1,1}(\mathbb{R})$. In order to establish the boundedness of $\left\|(I-\mathcal{C}_{v_{x}})^{-1}\right\|_{L^{2}\rightarrow L^{2}}$ and the Lipschitz continuity of the inverse scattering map, we give the following Lemma to characterize the uniform resolvent bound.
\begin{lemma}\label{urb}
Suppose that $v\in Z$, then for fixed $a\in\mathbb{R}$ and any bounded subset $B$ of $Z$, we have
\begin{equation}
\sup_{v\in B}\sup_{x\in[a,\infty)}\|(I-\mathcal{C}_{v_{x}})^{-1}\|_{\mathscr{B}\left(L^{2}\right)}<\infty,
\end{equation}
and the map
\begin{equation*}
v\mapsto\left\{x\mapsto (I-\mathcal{C}_{v_{x}})^{-1}\right\}
\end{equation*}
is Lipschitz continuous from $Z$ to $C\left([a,\infty);\mathscr{B}\left(L^{2}\right)\right)$.
\end{lemma}
\begin{proof}
The embedding
$H^{1,1}(\mathbb{R})\hookrightarrow H^{\frac{1}{2}+\varepsilon}(\mathbb{R})$
is compact, so we only need to verify the three conditions of the Proposition B.1 in \cite{Jenkins}

(i) It clearly follows from \eqref{cw1} that the map $(v,x)\mapsto\mathcal{C}_{v_{x}}$ is continuous from $Y\times\mathbb{R}$ to $\mathscr{B}(L^{2}(\mathbb{R}))$ and the estimate
\begin{equation*}
\sup_{x\in\mathbb{R}}\|\mathcal{C}_{v_{x}}-\mathcal{C}_{v_{x}'}\|_{\mathscr{B}(L^{2})}\leq\|v-v'\|_{Y}
\end{equation*}
holds.

(ii) For all $x\in\mathbb{R}$ and $v\in Y$, the existence of the resolvent $(I-\mathcal{C}_{v_{x}})^{-1}$ is guaranteed by Proposition \ref{prh1}.

(iii) Finally, let us prove that for each $v\in Y$, the estimate
\begin{equation}\label{iii}
\sup_{x\in[a,+\infty)}\|(I-\mathcal{C}_{v_{x}})^{-1}\|_{\mathscr{B}(L^{2})}<\infty
\end{equation}
holds. Let $\breve{v}_{\pm}\triangleq(v_{\pm})^{-1}$, then an operator $\mathcal{C}_{\breve{v}_{x}}$ similar to operator $\mathcal{C}_{v_{x}}$ can be defined as
\begin{equation*}
\mathcal{C}_{\breve{v}_{x}}\phi=\mathcal{C}_{\mathbb{R}}^{+}\left(\phi\left(I_{p+q}-\breve{v}_{x-}\right)\right)+\mathcal{C}_{\mathbb{R}}^{-}\left(\phi\left(\breve{v}_{x+}-I_{p+q}\right)\right).
\end{equation*}
Further, we define operator $T_{v_{x}}$ as
\begin{equation*}
T_{v_{x}}\phi:=\mathcal{C}_{\mathbb{R}}^{+}\left(\mathcal{C}_{\mathbb{R}}^{-}\phi(v_{x+}-v_{x-})\right)(I_{p+q}-\breve{v}_{x-})+\mathcal{C}_{\mathbb{R}}^{-}\left(\mathcal{C}_{\mathbb{R}}^{+}\phi(v_{x+}-v_{x-})\right)(\breve{v}_{x+}-I_{p+q}).
\end{equation*}
A direct calculation using the properties of the Cauchy projection operators $\mathcal{C}_{\mathbb{R}}^{\pm}$ yields
\begin{equation*}
I-T_{v_{x}}=(I-\mathcal{C}_{v_{x}})(I-\mathcal{C}_{\breve{v}_{x}}),
\end{equation*}
then
\begin{equation*}
(I-\mathcal{C}_{v_{x}})^{-1}=(I-\mathcal{C}_{\breve{v}_{x}})(I-T_{v_{x}})^{-1}.
\end{equation*}
It follows from the rational approximation in \cite{Zhou1989b} that
\begin{equation*}
\lim_{x\rightarrow +\infty}\|T_{v_{x}}\|_{L^{2}\rightarrow L^{2}}=0.
\end{equation*}
Choose $b\in\mathbb{R}$ such that $\forall x\geq b$, $\|T_{v_{x}}\|_{L^{2}\rightarrow L^{2}}<\frac{1}{2}$. Therefore, the operator $I-T_{v_{x}}$ is invertible for every $x\geq b$ and a bound on the inverse operator is given by
\begin{equation*}
\sup_{x\in[b,+\infty)}\|(I-T_{v_{x}})^{-1}\|_{L^{2}\rightarrow L^{2}}\leq \sum_{n=0}^{\infty}\frac{1}{2^{n}}.
\end{equation*}
Clearly, $\|(I-\mathcal{C}_{\breve{v}_{x}})\|_{L^{2}\rightarrow L^{2}}$ is bounded uniformly in $x$. Then for each $v\in Y$, we have
\begin{equation*}
\sup_{x\in[b,+\infty)}\|(I-\mathcal{C}_{v_{x}})^{-1}\|_{L^{2}\rightarrow L^{2}}<\infty.
\end{equation*}
When $a<b$, the estimate \eqref{iii} still holds because $\|(I-\mathcal{C}_{v_{x}})^{-1}\|_{L^{2}\rightarrow L^{2}}$ is bounded uniformly in $x\in[a,b]$.
\end{proof}

Before giving the estimates of the reconstructed potential, we first show that $M(x,t,k)$ is a fundamental solution of the Lax equations \eqref{2.1} and \eqref{2.2}.
\begin{proposition}\label{lax}
Suppose that $M(x,t,k)$ solves the RH problem \ref{RH1}, $U(x,t)$ and $P(x,t,k)$ are given by
\begin{equation}\label{qq1}
\begin{aligned}
U(x,t)&=-\frac{1}{2\pi}\operatorname{ad}\sigma_{3}\left(\int_{\mathbb{R}}\nu(x,t,k)\left(v_{x+}(t,k)-v_{x-}(t,k)\right)dk\right)\\
&=\begin{bmatrix}0_{p\times p}&\boldsymbol{Q}(x,t)\\ \boldsymbol{Q}^{\dag}(x,t)&0_{q\times q}\end{bmatrix},
\end{aligned}
\end{equation}
and
\begin{equation}\label{p}
P(x,t,k)=2kU(x,t)+i\sigma_{3}U_{x}(x,t)-i\sigma_{3}U^{2}(x,t),
\end{equation}
respectively, then $M(x,t,k)$ satisfies the Lax pair \eqref{x1}-\eqref{t1}.
\end{proposition}
\begin{proof}
Let us first define two linear operators $L$ and $N$ by
\begin{align}
LM&:=M_{x}+ik\operatorname{ad}\sigma_{3}(M)-UM,\label{l}\\
NM&:=M_{t}+2ik^{2}\operatorname{ad}\sigma_{3}(M)-kP_{1}M-P_{2}M.\label{n}
\end{align}
Differentiating both sides of \eqref{vx} with respect to $x$ and $t$ respectively, we obtain
\begin{align*}
\frac{\partial M_{+}}{\partial x}+ik\operatorname{ad}\sigma_{3}(M_{+})&=\left(\frac{\partial M_{-}}{\partial x}+ik\operatorname{ad}\sigma_{3}(M_{-})\right)v_{x},\\
\frac{\partial M_{+}}{\partial t}+2ik^{2}\operatorname{ad}\sigma_{3}(M_{+})&=\left(\frac{\partial M_{-}}{\partial t}+2ik^{2}\operatorname{ad}\sigma_{3}(M_{-})\right)v_{x},
\end{align*}
which imply that $L$ and $N$ satisfy \eqref{vx}, namely,
\begin{align}
\left(LM\right)_{+}&=\left(LM\right)_{-}v_{x},\label{ll}\\
\left(NM\right)_{+}&=\left(NM\right)_{-}v_{x}.\label{nn}
\end{align}
Substituting the asymptotic expansion
\begin{equation}\label{exp}
M(x,t,k)=I_{p+q}+\frac{M_{1}(x,t)}{k}+\frac{M_{2}(x,t)}{k^{2}}+O(\frac{1}{k^{3}}), \quad k\rightarrow \infty,
\end{equation}
into \eqref{l} and \eqref{n} respectively, we have
\begin{align*}
&LM=i\operatorname{ad}\sigma_{3}(M_{1})-U+O(\frac{1}{k}),\\
&NM=k\left(2i\operatorname{ad}\sigma_{3}(M_{1})-P_{1}\right)+\left(2i\operatorname{ad}\sigma_{3}(M_{2})-P_{1}M_{1}-P_{2}\right)+O(\frac{1}{k}).
\end{align*}
Actually, according to \eqref{re0} and \eqref{r1}, we can rewrite \eqref{qq1} in the equivalent form
\begin{equation}\label{aa}
U=i\operatorname{ad}\sigma_{3}(M_{1}),
\end{equation}
then $LM=O(\frac{1}{k})$ as $k\rightarrow \infty$. Combining this result and \eqref{ll}, we see that $LM$ satisfies the homogeneous version of the RH problem \ref{RH1}, thus $LM=0$ from the process of proving Proposition \ref{prh1}.

Clearly, the matrix function $P(x,t,k)$ given by \eqref{p} can be decomposed into
\begin{equation*}
P(x,t,k)=kP_{1}(x,t)+P_{2}(x,t),
\end{equation*}
where $P_{1}(x,t)=2U(x,t)$ and $P_{2}(x,t)=i\sigma_{3}U_{x}(x,t)-i\sigma_{3}U^{2}(x,t)$. It is straightforward to obtain
\begin{equation}\label{p1}
P_{1}(x,t)=2i\operatorname{ad}\sigma_{3}(M_{1})
\end{equation}
from equality \eqref{aa}. Substituting the expansion \eqref{exp} into $LM=0$ and considering the $O(\frac{1}{k})$ term, we have
\begin{equation}\label{aaa}
M_{1x}+ik\operatorname{ad}\sigma_{3}(M_{2})-UM_{1}=0.
\end{equation}
Let $M_{1}=M_{1}^{(off)}+M_{1}^{(diag)}$, where the superscripts represent the off-diagonal and the diagonal parts of the block matrix $M_{1}$. From \eqref{aa} and \eqref{aaa}, we have
\begin{equation*}
\begin{aligned}
M_{1x}^{(off)}&=-\frac{i}{2}\sigma_{3}U_{x},\\
M_{1x}^{(diag)}&=UM_{1}^{(off)}=\frac{i}{2}\sigma_{3}U^{2},
\end{aligned}
\end{equation*}
then $P_{2}(x,t)$ can be rewritten in the form
\begin{equation}\label{p2}
P_{2}(x,t)=-2M_{1x}=2ik\operatorname{ad}\sigma_{3}
(M_{2})-PM_{1}.
\end{equation}
Therefore, $NM=O(\frac{1}{k})$ as $k\rightarrow \infty$ by \eqref{p1} and \eqref{p2}. Similarly, it follows from \eqref{nn} and Proposition \ref{prh1} that $NM=0$. It is now obvious that the lemma holds.
\end{proof}

\begin{lemma}\label{fou}
(see \cite{Zhou1998}).
If $\boldsymbol{R}\in H^{1}(\mathbb{R})$, then for $x\geq0$, we have
\begin{align}
&\left\|\mathcal{C}_{\mathbb{R}}^{+}(\boldsymbol{R}e^{-2ixk})\right\|_{L^{2}}\leq\frac{1}{\sqrt{1+x^{2}}}\|\boldsymbol{R}\|_{H^{1}},\label{fou1}\\
&\left\|\mathcal{C}_{\mathbb{R}}^{-}(\boldsymbol{R}e^{2ixk})\right\|_{L^{2}}\leq\frac{1}{\sqrt{1+x^{2}}}\|\boldsymbol{R}\|_{H^{1}}.\label{fou2}
\end{align}
\end{lemma}
\begin{proposition}\label{H11}
If $\boldsymbol{R}\in H^{1,1}(\mathbb{R})$, then $\boldsymbol{Q}\in H^{1,1}(\mathbb{R})$ and the map $\boldsymbol{R}\mapsto \boldsymbol{Q}$ is Lipschitz continuous from $H^{1,1}(\mathbb{R})$ to $H^{1,1}(\mathbb{R})$.
\end{proposition}
\begin{proof}
For $x\geq0$, by $\nu(x,k)$ satisfying the integral equation \eqref{nu} and the definition of $\mathcal{C}_{v_{x}}$ in \eqref{cw1}, we can decompose the integral in the right-hand side of \eqref{r1} into
\begin{equation}\label{d1}
\int_{\mathbb{R}}\nu(x,k)\left(v_{x+}(k)-v_{x-}(k)\right)dk\triangleq f_{1}(x)+f_{2}(x)+f_{3}(x),
\end{equation}
where
\begin{equation*}
\begin{aligned}
&f_{1}(x)=\int_{\mathbb{R}}\left(v_{x+}(k)-v_{x-}(k)\right)dk,\\
&f_{2}(x)=\int_{\mathbb{R}}\left(\mathcal{C}_{v_{x}}I_{p+q}\right)\left(v_{x+}(k)-v_{x-}(k)\right)dk,\\
&f_{3}(x)=\int_{\mathbb{R}}\left(\mathcal{C}_{v_{x}}(I-\mathcal{C}_{v_{x}})^{-1}\mathcal{C}_{v_{x}}I_{p+q}\right)\left(v_{x+}(k)-v_{x-}(k)\right)dk.
\end{aligned}
\end{equation*}
First, we show that $\boldsymbol{Q}\in L^{2,1}(\mathbb{R}^{+})$. It is easy to see that $f_{1}(x)\in L^{2,1}(\mathbb{R}^{+})$ by Fourier transform and $f_{2}(x)$ makes no contribution to the reconstruction of $\boldsymbol{Q}$ since $f_{2}(x)$ is a block diagonal matrix. For $f_{3}(x)$, we set $g=(I-\mathcal{C}_{v_{x}})^{-1}\mathcal{C}_{v_{x}}I_{p+q}$.
By the expression of $v_{x\pm}$ in \eqref{vv}, we can use Lemmas \ref{urb} and \ref{fou} to estimate
\begin{equation}\label{g}
\|g\|_{L^{2}}\leq\left\|(I-\mathcal{C}_{v_{x}})^{-1}\right\|_{L^{2}\rightarrow L^{2}}\left\|\mathcal{C}_{v_{x}}I_{p+q}\right\|_{L^{2}}
\lesssim\frac{1}{\sqrt{1+x^{2}}}.
\end{equation}
Using Lemma \ref{fou} and \eqref{g}, we can show that
\begin{equation*}
\begin{aligned}
\left|\left(f_{3}\right)_{12}\right|&\leq\left|\int_{\mathbb{R}}\left(\mathcal{C}_{\mathbb{R}}^{-}g(v_{x+}-I_{p+q})\right)(I_{p+q}-v_{x-})dk\right|\\
&=\left|\int_{\mathbb{R}}\left(\mathcal{C}_{\mathbb{R}}^{-}g(v_{x+}-I_{p+q})\right)\mathcal{C}_{\mathbb{R}}^{+}(I_{p+q}-v_{x-})dk\right|\\
&\lesssim\left\|g\right\|_{L^{2}}\left\|\mathcal{C}_{\mathbb{R}}^{+}(I_{p+q}-v_{x-})\right\|_{L^{2}}\\
&\lesssim\frac{1}{1+x^{2}}.
\end{aligned}
\end{equation*}

Next, we prove that $\boldsymbol{Q}\in H^{1}(\mathbb{R}^{+})$.
It follows from \eqref{bbb} and Proposition \ref{lax} that the solution $\nu(x,k)$ of the integral equation \eqref{nu} satisfies the differential equation \eqref{x1}, then, we can computer that
\begin{equation*}
\frac{d}{dx}\left(\nu(v_{x+}-v_{x-})\right)=\left(-ik\operatorname{ad}\sigma_{3}+U\right)\left(\nu(v_{x+}-v_{x-})\right).
\end{equation*}
Integrating the both sides of the above equation, we have
\begin{equation*}
\frac{d}{dx}\int_{\mathbb{R}}\left(\nu(v_{x+}-v_{x-})\right)dk= H_{1}(x)+H_{2}(x),
\end{equation*}
where
\begin{equation*}
\begin{aligned}
&H_{1}(x)=\int_{\mathbb{R}}-i\operatorname{ad}\sigma_{3}\left(\nu e^{-ixk\operatorname{ad \sigma_{3}}}k(v_{+}-v_{-})\right)dk,\\
&H_{2}(x)=\int_{\mathbb{R}}U\nu(v_{x+}-v_{x-})dk.
\end{aligned}
\end{equation*}
The first integral $H_{1}(x)$ can be decomposed into three parts, as in \eqref{d1}. By $R\in H^{1,1}$, an argument similar to the one used in \eqref{d1} shows that $H_{1}\in L^{2}$. For $H_{2}(x)$, it is sufficient to prove $\boldsymbol{Q}\in L^{\infty}$, consequently, $H_{2}\in L^{2}$ follows immediately from the proceeding as in the proof of $\boldsymbol{Q}\in L^{2,1}$.

The estimate for $x\leq0$ can be analogously obtained by considering \eqref{ml} and the Lipschitz continuity of the map follows from the results of Lemma \ref{urb}.
\end{proof}

\subsection{Proof of Theorem \ref{thm} for defocusing case}\label{s36}

Summarizing the results of Sections \ref{s2} and \ref{s3}, we can prove the bijectivity results and existence of global solutions for defocusing $p\times q$ matrix NLS equation \eqref{mnls}.
\begin{proof}
The reflection coefficient $R$ is the main component of $v_{\pm}$ by  definition \eqref{vpm}, so the results of Propositions \ref{pro1} and \ref{H11} indicate that the maps \eqref{D11} and \eqref{I11} are Lipschitz continuous. Given initial date $\boldsymbol{Q}_{0}\in H^{1,1}(\mathbb{R})$, Proposition \ref{ct1} shows that the map $\boldsymbol{Q}_{0}\mapsto v_{\pm}-I_{p+q}$ is Lipschitz continuous from $H^{1,1}(\mathbb{R})$ to $C([-T,T], H^{1,1}(\mathbb{R}))$ for every $T>0$. It follows from Propositions \ref{prh1} and \ref{lax} that the $\boldsymbol{Q}(x,t)$ defined by \eqref{r1} solves \eqref{mnls}. From Proposition \ref{H11}, we have the map $\boldsymbol{Q}_{0}\mapsto \boldsymbol{Q}$ is Lipschitz continuous from $H^{1,1}(\mathbb{R})$ to $C\left([-T,T], H^{1,1}(\mathbb{R})\right)$ for every $T>0$.

The uniqueness of the solution can be obtained from group theory as well as Gronwall's inequality.
Suppose that the initial date $\boldsymbol{Q}_{01}, \boldsymbol{Q}_{02} \in H^{1,1}(\mathbb{R})$, denote the corresponding solution by $\boldsymbol{Q}_{1}$ and $\boldsymbol{Q}_{2}$ respectively, then we have
\begin{equation*}
\boldsymbol{Q}_{1}(t)-\boldsymbol{Q}_{2}(t)=S(t)(\boldsymbol{Q}_{01}-\boldsymbol{Q}_{02})+\int_{0}^{t}S(t-\tau)\left(F(\boldsymbol{Q}_{1})-F(\boldsymbol{Q}_{2})\right)(\tau)d\tau,
\end{equation*}
where $\{S(t)\}_{t\in\mathbb{R}}$ is the group of isometries generated by $i\partial_{xx}$ and $F(\boldsymbol{Q})=-2i\boldsymbol{Q}\boldsymbol{Q}^{\dag}\boldsymbol{Q}$. Then for any $0\leq t\leq T$ ($T>0$), we have
\begin{equation*}
\begin{aligned}
\|\boldsymbol{Q}_{1}(t)-\boldsymbol{Q}_{2}(t)\|_{H^{1}(\mathbb{R};\mathbb{C}^{p\times q})}
\leq&\|\boldsymbol{Q}_{01}-\boldsymbol{Q}_{02}\|_{H^{1}(\mathbb{R};\mathbb{C}^{p\times q})}\\
&+C_{T}\int_{0}^{t}\|\boldsymbol{Q}_{1}(\tau)-\boldsymbol{Q}_{2}(\tau)\|_{H^{1}(\mathbb{R};\mathbb{C}^{p\times q})}d\tau,
\end{aligned}
\end{equation*}
here we use the fact that $H^{1}$ is a Banach Algebra. From the Gronwall's inequality, we have
\begin{equation*}
\|\boldsymbol{Q}_{1}(t)-\boldsymbol{Q}_{2}(t)\|_{H^{1}(\mathbb{R};\mathbb{C}^{p\times q})}
\leq\|\boldsymbol{Q}_{01}-\boldsymbol{Q}_{02}\|_{H^{1}(\mathbb{R};\mathbb{C}^{p\times q})}e^{C_{T}t},
\end{equation*}
if $\boldsymbol{Q}_{01}=\boldsymbol{Q}_{02}$, then $\boldsymbol{Q}_{1}=\boldsymbol{Q}_{2}$ in $[0,T]$. The same is true for $-T\leq t\leq 0$, thus the solution is unique.
\end{proof}

\section{Inverse Scattering Map for Focusing Case}\label{s4}

In this section, we  study the  inverse scattering  transform  and further
prove the  well-posedness  in  $H^{1,1}(\mathbb{R})$  for the
 focusing   matrix NLS equation.

\subsection{Symmetry reduction}\label{s41}
When $\sigma=-1$, from \eqref{Q} we have
\begin{equation}\label{Q22}
U^{\dag}=-U.
\end{equation}
By \eqref{2.3}, \eqref{2+abcd} and the symmetry \eqref{Q22}, we have
\begin{equation}\label{m-1}
(m^{\pm})^{-1}(k)=(m^{\pm})^{\dag}(k^{*})
\end{equation}
and
\begin{equation}\label{s-1}
S^{-1}(k)=S^{\dag}(k^{*}).
\end{equation}
Further using the symmetry relation \eqref{s-1} we have
\begin{equation}\label{ca}
\boldsymbol{C}(k)\boldsymbol{A}^{-1}(k)=-\left(\boldsymbol{B}(k)\boldsymbol{D}^{-1}(k)\right)^{\dag}=-\boldsymbol{R}^{\dag}(k).
\end{equation}
Unlike the case of defocusing, the focusing case allows arbitrary numbers of poles and spectral singularities. This implies that $\boldsymbol{D}(k)$ and $\boldsymbol{A}(k)$ are not always invertible, so we must explore more properties of the matrix scattering date $\boldsymbol{D}(k)$ and $\boldsymbol{A}(k)$ to overcome this difficulty.
\begin{lemma}\label{lad}
Suppose that $\boldsymbol{Q}\in H^{1,1}(\mathbb{R})$, then we have
\begin{equation}\label{Ak}
\lim_{k\rightarrow\infty}\det \boldsymbol{D}(k)=\lim_{k\rightarrow\infty}\det \boldsymbol{A}(k)=1.
\end{equation}
If in addition, $\|U\|_{L^{1}}\ll1$, then $\det \boldsymbol{D}(k)$ and $\det \boldsymbol{A}(k)$ have no zeros in $\mathbb{C}^{-}\cup\mathbb{R}$ and $\mathbb{C}^{+}\cup\mathbb{R}$, respectively.
\end{lemma}
\begin{proof}
Using \eqref{2+s} and \eqref{2.5}, we can calculate that
\begin{equation}\label{A}
\boldsymbol{A}(k)=I_{p}+\int_{\mathbb{R}}\boldsymbol{Q}(y)m_{21}^{-}(y,k)dy,
\end{equation}
and
\begin{equation*}
m_{21}^{-}(x,k)=-\int_{-\infty}^{x}e^{-2i(x-y)k}\boldsymbol{Q}^{\dag}(y)m_{11}^{-}(y,k)dy.
\end{equation*}
By Lebesgue's dominated convergence theorem, we have $\lim_{k\rightarrow\infty}m_{21}^{-}(x,k)=0_{q\times p}$ and $\lim_{k\rightarrow\infty}\boldsymbol{A}(k)=I_{p}$. Thus \eqref{Ak} holds. From \eqref{A} and the Lipschitz continuity in Proposition \ref{lmab}, we obtain that for every $\varepsilon>0$ there is a $\delta>0$ such that $\left|\det \boldsymbol{A}(k)-1\right|<\varepsilon$ for all $k\in \mathbb{C}^{+}\cup\mathbb{R}$  when $\|U\|_{L^{1}}<\delta$. For example, if $\varepsilon=1/2$, then exist a $\delta>0$ such that $|\det \boldsymbol{A}(k)|>1/2$ when $\|U\|_{L^{1}}<\delta$. We can get the result for $\boldsymbol{D}(k)$ by a similar way.
\end{proof}
In the next subsection we will show how Lemma \ref{lad} plays a crucial role.

\subsection{Set-up of an   RH problem }\label{s42}
We still start by constructing a piecewise analytic matrix function from \eqref{2+abcd}, let
\begin{equation}\label{mxk}
M(x,k)=\left\{
\begin{aligned}
&\begin{bmatrix}m_{1}^{-}(x,k)\boldsymbol{A}^{-1}(k),&m_{2}^{+}(x,k)\end{bmatrix},\quad \text{Im}k>0,\\
&\begin{bmatrix}m_{1}^{+}(x,k),&m_{2}^{-}(x,k)\boldsymbol{D}^{-1}(k)\end{bmatrix},\quad \text{Im}k<0,
\end{aligned}
\right.
\end{equation}
we have
\begin{equation}\label{M}
M_{+}(x,k)=M_{-}(x,k)e^{-ixk\operatorname{ad}\sigma_{3}}\breve{v}(k),\quad k\in\mathbb{R},
\end{equation}
where
\begin{equation}\label{v}
\breve{v}(k)=\begin{bmatrix}I_{p}+\boldsymbol{R}(k)\boldsymbol{R}^{\dag}(k^{*})&-\boldsymbol{R}(k)\\
-\boldsymbol{R}^{\dag}(k^{*})&I_{q}\end{bmatrix}.
\end{equation}
From \eqref{m11a}-\eqref{m2d}, we still obtain that $M(x,k)$ satisfies the differential equation \eqref{2.3} and $M(x,k)\rightarrow I_{p+q} \ \text{as} \ x\rightarrow+\infty$. We must emphasize that $M(x,k)$ is well-defined only when the matrix scattering date $\boldsymbol{A}(k)$ and $\boldsymbol{D}(k)$ are invertible. Next, we will make full use of the properties of $\boldsymbol{A}(k)$ and $\boldsymbol{D}(k)$ presented in Lemma \ref{lad} and construct a new matrix-value function $\boldsymbol{M}(x,k)$ by Zhou's technique \cite{Zhou1998} to overcome this difficulty.

Choose $x_{0}\in\mathbb{R}$ such that $\|U\chi_{(x_{0},+\infty)}\|_{L^{1}}\ll1$ and denote the corresponding cut-off potential as $\boldsymbol{Q}_{x_{0}}=\boldsymbol{Q}\chi_{(x_{0},+\infty)}$. Replacing the potential $\boldsymbol{Q}$ in \eqref{2.3} by $\boldsymbol{Q}_{x_{0}}$, then the associated Jost functions $m^{(0)\pm}$ and a new piecewise analytic function $M^{(0)}(x,k)$ can be constructed in the same way as for $m^{\pm}$ and $M(x,k)$, respectively. Let $M^{(1)}(x,k)$ is the unique solution of the Volterra equation
\begin{equation}
M^{(1)}(x,k)=I_{p+q}+\int_{x_{0}}^{x}e^{i(y-x)k\operatorname{ad}\sigma_{3}}U(y)M^{(1)}(y,k)dy,
\end{equation}
and define
\begin{equation}
M^{(2)}(x,k):=M^{(1)}(x,k)e^{-i(x-x_{0})k\operatorname{ad}\sigma_{3}}M^{(0)}(x_{0},k).
\end{equation}
By uniqueness, we have $M^{(2)}(x,k)=M^{(0)}(x,k),\ \forall x\geq x_{0}$ and $M^{(2)}(x,k)\rightarrow I_{p+q}, $ as $ x \rightarrow +\infty$.
Choose a disc $B(0,S_{\infty})=\{z\in\mathbb{C}: |z|<S_{\infty}\}$ so large that $\det \boldsymbol{D}(k)$ have no zeros in $\mathbb{C}\setminus B(0,S_{\infty})$, which can be guaranteed by \eqref{Ak}. Define the contour
\begin{equation*}
\Gamma=\mathbb{R}\cup\Gamma_{\infty},\quad \Gamma_{\infty}=\{z\in\mathbb{C}:|z|=S_{\infty}\},
\end{equation*}
oriented as in Figure \ref{f1}.
The positive region and negative region of $\Gamma$ are denoted as
\begin{equation}\label{omega}
\Omega_{+}=\Omega_{1}\cup\Omega_{4}\quad \text{and} \quad \Omega_{-}=\Omega_{2}\cup\Omega_{3},
\end{equation}
respectively. A new piecewise analytic matrix-value function $\boldsymbol{M}(x,k)$ can be defined as
\begin{equation}
\boldsymbol{M}(x,k)=
\left\{
\begin{aligned}
&M(x,k),& \quad &k\in\mathbb{C}\setminus (B(0,S_{\infty})\cup\mathbb{R}), \\
&M^{(2)}(x,k),& \quad &k\in B(0,S_{\infty})\setminus\mathbb{R},
\end{aligned}
\right.
\end{equation}
which is always well-defined. It follows from the construction process that $\boldsymbol{M}(x,k)$ also satisfies \eqref{2.3} and $\boldsymbol{M}(x,k)\rightarrow I_{p+q}\  \text{as} \ x\rightarrow+\infty$.

Let us calculate the jump matrix $V(k)$ for $\boldsymbol{M}(x,k)$, which is tricky because $V(k)$ is a $(p+q)\times(p+q)$ matrix function and the associated scattering date are all high-order matrices. For the $2\times 2$ AKNS system \cite{Zhou1998}, it is straightforward to compute the inverse matrix of a $2\times2$ matrix, and some additional conjugate symmetry between the scalar scattering date can be obtained to simplify calculation.
For the $3\times 3$ AKNS system \cite{LiuJQ}, the inverse matrix of the $3\times3$ matrix can also be obtained with the help of the adjoint matrix.
However, for the $(p+q)\times(p+q)$ AKNS system, we cannot directly calculate the inverse matrix of a $(p+q)\times(p+q)$ matrix. Fortunately, we find that this difficulty can be solved by considering the $(p+q)\times(p+q)$ matrix as the block matrix and making full use of \eqref{2+abcd}, \eqref{m-1} and \eqref{s-1}.

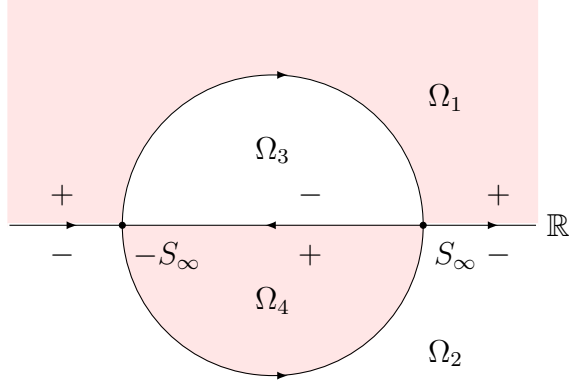
\begin{figure}
\begin{center}
\begin{tikzpicture}
\draw [fill=pink,ultra thick,color=red!10](-3.5,0) rectangle (3.5,3);
 \draw [fill=pink,ultra thick,color=white!10] (-3.5,0) rectangle (3.5,-2.2);
\filldraw[color=white!10](0,0)-- (2,0) arc (0:180:2);
\filldraw[color=red!10](0,0)-- (2,0) arc (0:-180:2);
\draw(0,0)circle(2cm);
\draw[fill,black] (2,0) circle [radius=0.04];
\draw[fill,black] (-2,0) circle [radius=0.04];
\draw [ ](-3.5,0)--(3.5,0)  node[right, scale=1] {$\mathbb{R}$};
\draw [-latex](2.95,0)--(3,0);
\draw [-latex](-2.65,0)--(-2.6,0);
\draw [-latex](0,0)--(-0.1,0);
\draw [-latex]  (0.15,2)--(0.2,2);
\draw [-latex]  (0.15,-2)--(0.2,-2);
\node at (3, 0.4 )  {$+$};
\node at (3, -0.4 )  {$-$};
\node at (-2.8, 0.4 )  {$+$};
\node at (-2.8, -0.4 )  {$-$};
\node at (0.5, 0.4 )  {$-$};
\node at (0.5, -0.4 )  {$+$};
\node [right] at (-2,-0.4)  {$-S_{\infty}$};
\node [right] at (2,-0.4)  {$S_{\infty}$};
\node at (2.3,1.7)  {$\Omega_{1}$};
\node at (2.3,-1.7)  {$\Omega_{2}$};
\node at (0,1)  {$\Omega_{3}$};
\node at (0,-1)  {$\Omega_{4}$};
\end{tikzpicture}
\end{center}
\caption{The oriented contour $\Gamma$.}
\label{f1}
\end{figure}

\begin{proposition}\label{pv}
The jump matrix $V(k)$ of the piecewise analytic function $\boldsymbol{M}(x,k)$ along the contour $\Gamma$ satisfying
\begin{equation}
\boldsymbol{M}_{+}(x,k)=\boldsymbol{M}_{-}(x,k)e^{-ixk\operatorname{ad}\sigma_{3}}V(k)
\end{equation}
can be given as follows:
\begin{enumerate}[label=(\roman*)]
  \item For $\Gamma_{1}=(-\infty,-S_{\infty})\cup(S_{\infty},+\infty)$,
  \begin{equation}\label{V1}
  V(k)=\begin{bmatrix}I_{p}+ \boldsymbol{R}(k)\boldsymbol{R}^{\dag}(k)&-\boldsymbol{R}(k)\\ -\boldsymbol{R}^{\dag}(k)&I_{q}\end{bmatrix}.
  \end{equation}
  \item For $\Gamma_{2}=\Gamma_{\infty}\cap\mathbb{C}^{+}$,
  \begin{equation}\label{V2}
  V(k)=\begin{bmatrix}I_{p}&0_{p\times q}\\e^{-2ix_{0}k}(\boldsymbol{D}_{0}^{-1})^{\dag}(k^{*})m_{21}^{-}(x_{0},k)\boldsymbol{A}^{-1}(k)&I_{q}\end{bmatrix}.
  \end{equation}
  \item For $\Gamma_{3}=(-S_{\infty},S_{\infty})$,
  \begin{equation}\label{V3}
  V(k)=\begin{bmatrix}I_{p}&\boldsymbol{R}_{0}(k)\\ \boldsymbol{R}_{0}^{\dag}(k) &I_{q}+ \boldsymbol{R}_{0}^{\dag}(k)\boldsymbol{R}_{0}(k)\end{bmatrix}.
  \end{equation}
  \item For $\Gamma_{4}=\Gamma_{\infty}\cap\mathbb{C}^{-}$,
  \begin{equation}\label{V4}
  V(k)=\begin{bmatrix}I_{p}&e^{2ix_{0}k}(\boldsymbol{A}^{-1})^{\dag}(k^{*})(m_{21}^{-})^{\dag}(x_{0},k^{*})\boldsymbol{D}_{0}^{-1}(k)\\0_{q\times p}&I_{q}\end{bmatrix}.
  \end{equation}
\end{enumerate}
Here, the subscript ``0" denotes the scattering date or reflection coefficient associated with the potential $\boldsymbol{Q}_{x_{0}}$.
\end{proposition}
\begin{proof}
Let $V_{j}$ denote the restriction of $V$ to $\Gamma_{j}$ with $1\leq j\leq4$. The \eqref{V1} for $V_{1}$ can be obtained straightforwardly from equations \eqref{mxk}-\eqref{v}. From the expression of $V_{1}$ we have
\begin{equation}
  V_{3}(k)=\begin{bmatrix}I_{p}+ \boldsymbol{R}_{0}(k)\boldsymbol{R}_{0}^{\dag}(k)&-\boldsymbol{R}_{0}(k)\\ -\boldsymbol{R}_{0}^{\dag}(k)&I_{q}\end{bmatrix}^{-1}=\begin{bmatrix}I_{p}&\boldsymbol{R}_{0}(k)\\ \boldsymbol{R}_{0}^{\dag}(k) &I_{q}+ \boldsymbol{R}_{0}^{\dag}(k)\boldsymbol{R}_{0}(k)\end{bmatrix}.
\end{equation}

So next we mainly describe how to obtain $V_{2}(k)$ and $V_{4}(k)$. In fact, it is sufficient to consider $x=x_{0}$. It follows from the Volterra integral equations \eqref{2.5} that
\begin{equation}\label{m0}
m^{(0)-}(x_{0},k)=I_{p+q}, \quad m^{(0)+}(x_{0},k)=m^{+}(x_{0},k).
\end{equation}
Applying the symmetry relation \eqref{m-1} and \eqref{m0} to \eqref{2+abcd}, we have
\begin{equation}\label{ad0}
\boldsymbol{A}_{0}(k)=(m_{11}^{+})^{\dag}(x_{0},k^{*}), \quad \boldsymbol{D}_{0}(k)=(m_{22}^{+})^{\dag}(x_{0},k^{*}).
\end{equation}
Make use of \eqref{2+abcd}, \eqref{m-1}, \eqref{m0} and \eqref{ad0}, we calculate that
\begin{equation*}
\begin{aligned}
&e^{-ix_{0}k\operatorname{ad}\sigma_{3}}V_{2}(k)
=\left(M^{(0)}(x_{0},k)\right)^{-1}M(x_{0},k)\\
=&\begin{bmatrix}\boldsymbol{A}_{0}^{-1}(k)&m_{12}^{+}(x_{0},k)\\0_{q\times p}&m_{22}^{+}(x_{0},k)\end{bmatrix}^{-1}\begin{bmatrix}m_{11}^{-}(x_{0},k)\boldsymbol{A}^{-1}(k)&m_{12}^{+}(x_{0},k)\\m_{21}^{-}(x_{0},k)\boldsymbol{A}^{-1}(k)&m_{22}^{+}(x_{0},k)\end{bmatrix}\\
=&\begin{bmatrix}(m_{11}^{+})^{\dag}(x_{0},k^{*})&(m_{21}^{+})^{\dag}(x_{0},k^{*})\\0_{q\times p}&(m_{22}^{+})^{-1}(x_{0},k)\end{bmatrix}\begin{bmatrix}m_{11}^{-}(x_{0},k)\boldsymbol{A}^{-1}(k)&m_{12}^{+}(x_{0},k)\\m_{21}^{-}(x_{0},k)\boldsymbol{A}^{-1}(k)&m_{22}^{+}(x_{0},k)\end{bmatrix}\\
=&\begin{bmatrix}I_{p}&0_{p\times q}\\(\boldsymbol{D}_{0}^{-1})^{\dag}(k^{*})m_{21}^{-}(x_{0},k)\boldsymbol{A}^{-1}(k)&I_{q}\end{bmatrix},
\end{aligned}
\end{equation*}
and
\begin{equation*}
\begin{aligned}
&e^{-ix_{0}k\operatorname{ad}\sigma_{3}}V_{4}(k)=M^{-1}(x_{0},k)M^{(0)}(x_{0},k)\\
=&\begin{bmatrix}m_{11}^{+}(x_{0},k)&m_{12}^{-}(x_{0},k)\boldsymbol{D}^{-1}(k)\\m_{21}^{+}(x_{0},k)&m_{22}^{-}(x_{0},k)\boldsymbol{D}^{-1}(k)\end{bmatrix}^{-1}\begin{bmatrix}m_{11}^{+}(x_{0},k)&0_{p\times q}\\m_{21}^{+}(x_{0},k)&\boldsymbol{D}_{0}^{-1}(k)\end{bmatrix}\\
=&\begin{bmatrix}-(\boldsymbol{A}^{-1})^{\dag}(k^{*})m_{21}^{-}(x_{0},k^{*})m_{22}^{-}(x_{0},k)(m_{12}^{-})^{-1}(x_{0},k)&(\boldsymbol{A}^{-1})^{\dag}(k^{*})m_{21}^{-}(x_{0},k^{*})\\(m_{12}^{+})^{\dag}(x_{0},k^{*})&(m_{22}^{+})^{\dag}(x_{0},k^{*})\end{bmatrix}\\
&\times\begin{bmatrix}m_{11}^{+}(x_{0},k)&0_{p\times q}\\m_{21}^{+}(x_{0},k)&\boldsymbol{D}_{0}^{-1}(k)\end{bmatrix}\\
=&\begin{bmatrix}I_{p}&(\boldsymbol{A}^{-1})^{\dag}(k^{*})(m_{21}^{-})^{\dag}(x_{0},k^{*})\boldsymbol{D}_{0}^{-1}(k)\\0_{q\times p}&I_{q}\end{bmatrix}.
\end{aligned}
\end{equation*}
Although the computation of the inverse matrix of $M(x_{0},k)$ is more difficult than that of $M^{(0)}(x_{0},k)$, the result \eqref{V4} can still be obtained by repeatedly using the symmetry relation \eqref{m-1}. Here we omit the redundant calculation process.
\end{proof}

In this work, the contour $\Gamma=\bigcup_{j=1}^{4}\Gamma_{j}$ is a piecewise smooth curve, where $\Gamma_{j}$ is non-self-intersections. We denote $f\in H^{k}(\Gamma)$ if on each smooth piece $f$ is $H^{k}$.
With the definition \eqref{space} of the spaces $H^{k,j}(\partial\Omega)$ and \eqref{omega}, we can describe the regularity of $V(k)$ on $\partial\Omega_{\pm}$ by the following Proposition.

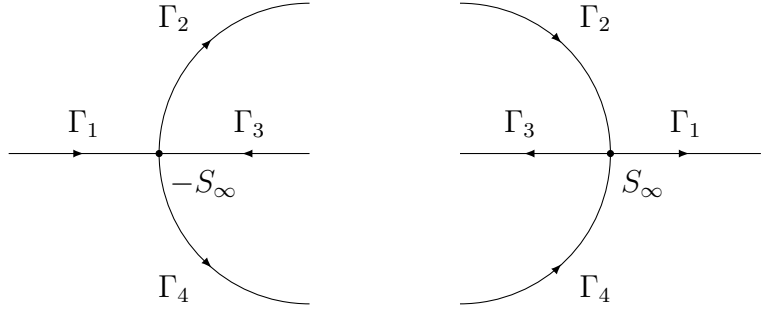
\begin{figure}
\begin{center}
\begin{tikzpicture}
\draw (1,0)--(5,0);
\draw (1,2) arc (90:-90:2);
\draw[fill,black] (3,0) circle [radius=0.04];
\draw [-latex](4,0)--(4.05,0);
\draw [-latex](1.90,0)--(1.85,0);
\draw [-latex](2.3,1.52)--(2.35,1.47);
\draw [-latex](2.3,-1.52)--(2.35,-1.47);
\node [right] at (3,-0.4)  {$S_{\infty}$};
\node at (4,0.4)  {$\Gamma_{1}$};
\node at (1.8,0.4)  {$\Gamma_{3}$};
\node at (2.8,1.8)  {$\Gamma_{2}$};
\node at (2.8,-1.8)  {$\Gamma_{4}$};

\draw (-5,0)--(-1,0);
\draw (-1,2) arc (90:270:2);
\draw[fill,black] (-3,0) circle [radius=0.04];
\draw [-latex](-4.05,0)--(-4,0);
\draw [-latex](-1.85,0)--(-1.9,0);
\draw [-latex](-2.35,1.47)--(-2.3,1.52);
\draw [-latex](-2.35,-1.47)--(-2.3,-1.52);
\node [right] at (-3,-0.4)  {$-S_{\infty}$};
\node at (-4,0.4)  {$\Gamma_{1}$};
\node at (-1.8,0.4)  {$\Gamma_{3}$};
\node at (-2.8,1.8)  {$\Gamma_{2}$};
\node at (-2.8,-1.8)  {$\Gamma_{4}$};
\end{tikzpicture}
\end{center}
\caption{Matching conditions at $\pm S_{\infty}$.}
\label{f2}
\end{figure}

\begin{proposition}\label{pvpm}
Suppose that $\boldsymbol{Q}\in H^{1,1}(\mathbb{R})$, then the jump matrix $V(k)$  in Proposition \ref{pv} admits a triangular factorization \begin{equation}
V(k)=V_{-}^{-1}(k)V_{+}(k), \quad k\in \Gamma,
\end{equation}
where $V_{\pm}(k)-I_{p+q}$ and $V_{\pm}^{-1}(k)-I_{p+q}\in H^{1,1}(\partial\Omega_{\pm})$.
\end{proposition}
\begin{proof}
We first show that $V(k)-I_{p+q}, V^{-1}(k)-I_{p+q}\in H^{1}(\Gamma)$.
From Proposition \ref{pro1}, we have $\boldsymbol{R}(k)\in H^{1,1}(\mathbb{R})$ and $\boldsymbol{R}_{0}(k)\in H^{1}(\mathbb{R})$ and hence $V_{j}(k)-I_{p+q}, V_{j}^{-1}(k)-I_{p+q}\in H^{1}(\Gamma_{j}),\ j=1,3$.
By Propositions \ref{lm} and \ref{lmab}, we obtain $V_{j}(k)-I_{p+q}, V_{j}^{-1}(k)-I_{p+q}\in H^{1}(\Gamma_{j}),\ j=2,4,$ since $\Gamma_{j}, j=2,4$ can be parameterized as a subset of $\mathbb{R}$. Next, we verify the matching conditions at $\pm S_{\infty}$. By using the relation \eqref{2+abcd}, the value of the jump matrix $V_{2}(k)$ at point $\pm S_{\infty}$ can be rewritten as
\begin{align*}
  V_{2}(\pm S_{\infty})=&\begin{bmatrix}I_{p}&0_{p\times q}\\e^{-2ix_{0}(\pm S_{\infty})}(\boldsymbol{D}_{0}^{-1})^{\dag}(\pm S_{\infty})m_{21}^{-}(x_{0},\pm S_{\infty})\boldsymbol{A}^{-1}(\pm S_{\infty})&I_{q}\end{bmatrix}\\
  =&\begin{bmatrix}I_{p}&0_{p\times q}\\e^{-2ix_{0}(\pm S_{\infty})}(\boldsymbol{D}_{0}^{-1})^{\dag}(\pm S_{\infty})m_{21}^{+}(x_{0},\pm S_{\infty})&I_{q}\end{bmatrix}\times\\
  &\begin{bmatrix}I_{p}&0_{p\times q}\\(\boldsymbol{D}_{0}^{-1})^{\dag}(\pm S_{\infty})m_{22}^{+}(x_{0},\pm S_{\infty})\boldsymbol{C}(\pm S_{\infty})\boldsymbol{A}^{-1}(\pm S_{\infty})&I_{q}\end{bmatrix}\\
  =&\begin{bmatrix}I_{p}&0_{p\times q}\\(\boldsymbol{D}_{0}^{-1})^{\dag}(\pm S_{\infty})\boldsymbol{B}_{0}^{\dag}(\pm S_{\infty})&I_{q}\end{bmatrix}\begin{bmatrix}I_{p}&0_{p\times q}\\\boldsymbol{C}(\pm S_{\infty})\boldsymbol{A}^{-1}(\pm S_{\infty})&I_{q}\end{bmatrix}\\
  =&\begin{bmatrix}I_{p}&0_{p\times q}\\ \boldsymbol{R}_{0}^{\dag}(\pm S_{\infty})&I_{q}\end{bmatrix}\begin{bmatrix}I_{p}&0_{p\times q}\\-\boldsymbol{R}^{\dag}(\pm S_{\infty})&I_{q}\end{bmatrix}.
\end{align*}
Similarly, we have
\begin{align*}
V_{4}(\pm S_{\infty})=\begin{bmatrix}I_{p}&-\boldsymbol{R}(\pm S_{\infty})\\0_{q\times p}&I_{q}\end{bmatrix}\begin{bmatrix}I_{p}&\boldsymbol{R}_{0}(\pm S_{\infty})\\0_{q\times p}&I_{q}\end{bmatrix}.
\end{align*}
Combining \eqref{V1}, \eqref{V3} and the above results, we can calculate that
\begin{equation*}
V_{1}^{-1}(\pm S_{\infty})V_{4}(\pm S_{\infty})=V_{2}^{-1}(\pm S_{\infty})V_{3}(\pm S_{\infty}),
\end{equation*}
which means
\begin{equation*}
V_{1}V_{2}^{-1}V_{3}V_{4}^{-1}=I_{p+q}+o(1), \quad \text{as} \ k\rightarrow +S_{\infty},
\end{equation*}
and
\begin{equation*}
V_{3}^{-1}V_{2}V_{1}^{-1}V_{4}=I_{p+q}+o(1), \quad \text{as} \ k\rightarrow -S_{\infty}.
\end{equation*}
A more specific description is given in Figure \ref{f2}. Therefore, the jump matrix $V(k)$ defined on $\Gamma$ satisfies the matching conditions according to \cite{Zhou1999} and Theorem 2.56 in \cite{Trogdon}. Finally, it is important to emphasize that $V_{+}\upharpoonright_{\partial\Omega_{1}}-I_{p+q}
$
and $V_{-}\upharpoonright_{\partial\Omega_{3}}-I_{p+q}$ are strictly lower triangular, while $V_{-}\upharpoonright_{\partial\Omega_{2}}-I_{p+q}$
and $V_{+}\upharpoonright_{\partial\Omega_{4}}-I_{p+q}$ are strictly upper triangular.
\end{proof}

In order to estimate the decay and smoothness property of the  potential on the negative half-line of $x$, we can reselect $x_{0}$ such that the cut-off potentials $\boldsymbol{Q}_{x_{0}}=\boldsymbol{Q}\chi_{(x_{0},+\infty)}$ and $\widetilde{\boldsymbol{Q}}_{x_{0}}=\boldsymbol{Q}\chi_{(-\infty,-x_{0})}$ satisfy the conditions $\|U\chi_{(x_{0},+\infty)}\|_{L^{1}}\ll1$ and $\|U\chi_{(-\infty, -x_{0})}\|_{L^{1}}\ll1$, respectively.
Repeating the process of constructing $\boldsymbol{M}(x,k)$, we can use the $\widetilde{\boldsymbol{Q}}_{x_{0}}$ and $\widetilde{M}(x,k)$ defined in \eqref{ml} to establish a piecewise analytic matrix-value function $\widetilde{\boldsymbol{M}}(x,k)$ which satisfies
\begin{equation}\label{tim}
\widetilde{\boldsymbol{M}}(x,k)\rightarrow I_{p+q} \quad \text{as} \quad x\rightarrow -\infty.
\end{equation}
We define the auxiliary matrix $\boldsymbol{\delta}(k)$ by
\begin{equation*}
\boldsymbol{\delta}(k)=\boldsymbol{M}^{-1}(x,k) \widetilde{\boldsymbol{M}}(x,k).
\end{equation*}
Then the jump matrix $\widetilde{V}(k)$ for $\widetilde{\boldsymbol{M}}(x,k)$ can be given by
\begin{equation}\label{tiv}
\widetilde{V}(k)=\boldsymbol{\delta}_{-}^{-1}(k)V(k)\boldsymbol{\delta}_{+}(k),
\end{equation}
which satisfies the following properties similar to those of the Proposition \ref{pvpm}:
\begin{proposition}
Suppose that $\boldsymbol{Q}\in H^{1,1}(\mathbb{R})$, then the jump matrix $\widetilde{V}(k)$ admits a triangular factorization
\begin{equation}
\widetilde{V}(k)=\widetilde{V}_{-}^{-1}(k)\widetilde{V}_{+}(k), \quad k\in \Gamma,
\end{equation}
where $\widetilde{V}_{\pm}(k)-I_{p+q}$ and $\widetilde{V}_{\pm}^{-1}(k)-I_{p+q}\in H^{1,1}(\partial\Omega_{\pm})$.
\end{proposition}
\begin{remark}
Notice that the triangular properties of the matrices $V_{\pm}(k)-I_{p+q}$ and $\widetilde{V}_{\pm}(k)-I_{p+q}$ are completely opposite, that is,  $\widetilde{V}_{+}\upharpoonright_{\partial\Omega_{1}}-I_{p+q}$
and $\widetilde{V}_{-}\upharpoonright_{\partial\Omega_{3}}-I_{p+q}$ are strictly upper triangular, while $\widetilde{V}_{-}\upharpoonright_{\partial\Omega_{2}}-I_{p+q}$
and $\widetilde{V}_{+}\upharpoonright_{\partial\Omega_{4}}-I_{p+q}$  are strictly lower triangular. This difference makes $V$ and $\widetilde{V}$ contribute accordingly to the estimates of the potential on the positive and negative half-lines, respectively.
\end{remark}

\subsection{Time evolution of the matrix scattering data}\label{s43}
Given the fundamental solutions $\boldsymbol{M}_{\pm}(x,t,k)$ of the Lax pair \eqref{x1}-\eqref{t1}
satisfying
\begin{equation}
\boldsymbol{M}_{+}(x,t,k)=\boldsymbol{M}_{-}(x,t,k)e^{-ixk\operatorname{ad}\sigma_{3}}V(t,k).
\end{equation}
Substituting
\begin{equation*}
\boldsymbol{M}_{-}(x,t,k)e^{-ixk\sigma_{3}}\ \text{and} \ \boldsymbol{M}_{-}(x,t,k)e^{-ixk\sigma_{3}}V(t,k)
\end{equation*}
into \eqref{t1}, and using the fact that $\boldsymbol{M}_{\pm}(x,t,k)\rightarrow I_{p+q}$ as $x\rightarrow+\infty$, we obtain
\begin{equation}
V(t,k)=e^{-2itk^{2}\operatorname{ad}\sigma_{3}
}V(0,k)
\end{equation}
and
\begin{equation}\label{vt}
V_{\pm}(t,k)=e^{-2itk^{2}\operatorname{ad}\sigma_{3}}V_{\pm}(0,k).
\end{equation}
Therefore the results of Proposition \ref{pvpm} can be extended to $V_{\pm}(t,k)$ as shown in the following Proposition.
\begin{proposition}\label{ct}
If $\boldsymbol{Q}_{0}\in H^{1,1}(\mathbb{R})$, then for every $t\in[-T,T],\ T>0$, we have $V_{\pm}(t,k)-I_{p+q}\in H^{1,1}(\partial\Omega_{\pm})$. Moreover, the mapping $\boldsymbol{Q}_{0}\mapsto V_{\pm}(t,k)-I_{p+q}$ is Lipschitz continuous from $H^{1,1}(\mathbb{R})$ to $C([-T,T], H^{1,1}(\partial\Omega_{\pm}))$ for every $T>0$.
\end{proposition}
\begin{remark}
We can prove Proposition \ref{ct} by a similar proof way for Proposition \ref{ct1}. Although we need to show that $V_{\pm}(t,k)$ is uniformly continuous with respect to time $t$ in the sense of the norm $H^{1,1}(\partial\Omega_{\pm})$, it is sufficient to prove the statement on $k\in\mathbb{R}$ since the others can be parameterized onto $\mathbb{R}$. The results of Proposition \ref{ct} hold for $\widetilde{V}_{\pm}(t,k)$ as well.
\end{remark}

\subsection{Unique solvability of the RH problem}\label{s44}
For the case $\sigma=-1$: focusing matrix NLS equation, we consider the following RH problem:
\begin{rhp}\label{RH2}
Find a $(p+q)\times(p+q)$ matrix-valued function $\boldsymbol{M}(x,k)$ that satisfies the following conditions:
\begin{itemize}
  \item Analyticity: $\boldsymbol{M}(x,k)$ is analytic for $k\in\mathbb{C}\setminus\Gamma$.
  \item Jump condition:  $\boldsymbol{M}$ has continuous boundary values $\boldsymbol{M}_{\pm}(x,k)$ as $k'\rightarrow k$ from $\Omega_{\pm}$ to $\Gamma$ and
      \begin{equation}
      \boldsymbol{M}_{+}(x,k)=\boldsymbol{M}_{-}(x,k)V_{x}(k),\quad k\in\Gamma,
      \end{equation}
      where $V_{x}(k)=e^{-ixk\operatorname{ad \sigma_{3}}}V(k)$ and $V(k)$ is given by \eqref{V1}-\eqref{V4}.
  \item Normalization: $\boldsymbol{M}(x,k)=I_{p+q}+\mathcal{O}(k^{-1})$ as $k\rightarrow\infty$.
\end{itemize}
\end{rhp}
\begin{proposition}\label{prh2}
Suppose that $V_{\pm}(k)$ are given by Proposition \ref{pvpm} and satisfy $V_{\pm}(k)-I_{p+q}\in H^{\frac{1}{2}+\varepsilon}(\Gamma)$, $\varepsilon>0$, then there exists a unique solution to the RH problem \ref{RH2} for every $x\in\mathbb{R}$.
\end{proposition}
\begin{proof}
It follows from Proposition \ref{pvpm} that $V_{x}(k)$ admits the factorization
\begin{equation*}
V_{x}=\left(e^{-ixk\operatorname{ad \sigma_{3}
}}V_{-}^{-1}\right) \left(e^{-ixk\operatorname{ad \sigma_{3}}}V_{+}\right)\triangleq V_{x-}^{-1}V_{x+},
\end{equation*}
then we define
\begin{equation*}
W_{x}^{+}:=V_{x+}-I_{p+q}, \quad W_{x}^{-}:=I_{p+q}-V_{x-}.
\end{equation*}
Similarly, we only need to prove that the following Beals-Coifman integral equation for RH Problem \ref{RH2} is uniquely solvable
\begin{equation}\label{mu}
\mu(x,k)=I_{p+q}+(\mathcal{C}_{W_{x}}\mu)(x,k),
\end{equation}
where
\begin{equation}
\mathcal{C}_{W_{x}}\mu=\mathcal{C}_{\Gamma}^{+}(\mu W_{x}^{-})+\mathcal{C}_{\Gamma}^{-}(\mu W_{x}^{+}),
\end{equation}
and
\begin{equation*}
\mu=\boldsymbol{M}_{+}(I_{p+q}+W_{x}^{+})^{-1}=\boldsymbol{M}_{-}(I_{p+q}-W_{x}^{-})^{-1}.
\end{equation*}
Here, $C_{\Gamma}^{\pm}$ are the Cauchy integral operators on $L^{2}(\Gamma)$, see \cite{Zhou1989b}. By the Proposition 4.1 and 4.2 in \cite{Zhou1989b}, we obtain $(I-\mathcal{C}_{W_{x}})$ is a Fredholm operator with zero index since
$W_{x}^{+}$ and $W_{x}^{-}$ allow the uniform rational approximations.

Let us verify that jump matrix $V(k)$ given by Proposition \ref{pv} satisfies the conditions of Theorem 9.3 in \cite{Zhou1989b}.
It is easy to show that $V_{2}(k)=V_{4}^{\dag}(k^{*}),\ k\in \Gamma_{2}$ from  \eqref{V2} and \eqref{V4}.
For $V_{1}(k)$, we have
\begin{equation*}
\begin{aligned}
V_{1}+V_{1}^{\dag}=2\begin{bmatrix}I_{p}+\boldsymbol{R}\boldsymbol{R}^{\dag}&-\boldsymbol{R}\\-\boldsymbol{R}^{\dag}&I_{q}\end{bmatrix}=2\begin{bmatrix}I_{p}&-\boldsymbol{R}\\0_{q\times p}&I_{q}\end{bmatrix}\begin{bmatrix}I_{p}&-\boldsymbol{R}\\0_{q\times p}&I_{q}\end{bmatrix}^{\dag},
\end{aligned}
\end{equation*}
clearly the matrix $\begin{bmatrix}I_{p}&-\boldsymbol{R}\\0_{q\times p}&I_{q}\end{bmatrix}$ is invertible. Similarly, for $V_{3}(k)$ we have
\begin{equation*}
\begin{aligned}
V_{3}+V_{3}^{\dag}=2\begin{bmatrix}I_{p}&\boldsymbol{R}_{0}\\ \boldsymbol{R}_{0}^{\dag} &I_{q}+ \boldsymbol{R}_{0}^{\dag}\boldsymbol{R}_{0}(k)\end{bmatrix}=2\begin{bmatrix}I_{p}&0_{p\times q}\\ \boldsymbol{R}_{0}^{\dag} &I_{q}\end{bmatrix}\begin{bmatrix}I_{p}&0_{p\times q}\\ \boldsymbol{R}_{0}^{\dag} &I_{q}\end{bmatrix}^{\dag},
\end{aligned}
\end{equation*}
where the matrix $\begin{bmatrix}I_{p}&0_{p\times q}\\ \boldsymbol{R}_{0}^{\dag} &I_{q}\end{bmatrix}$ is invertible. Therefore $V(k)+V^{\dag}(k)$ is positive definite for $k\in\mathbb{R}$. Then from Theorem 9.3 in  \cite{Zhou1989b} we obtain $\ker_{L^{2}(\Gamma)}(I-\mathcal{C}_{W_{x}})=0$ and $(I-\mathcal{C}_{W_{x}})$ is invertible in $L^{2}(\Gamma)$. Furthermore, the solution to the RH problem \ref{RH2} is
given by
\begin{equation}\label{mc}
\begin{aligned}
\boldsymbol{M}(x,k)&=I_{p+q}+\mathcal{C}_{\Gamma}\left(\mu(W_{x}^{+}+W_{x}^{-})\right)(x,k)\\
&=I_{p+q}+\frac{1}{2\pi i}\int_{\Gamma}\frac{\mu(x,s)\left(W_{x}^{+}(s)+W_{x}^{-}(s)\right)}{s-k}ds.
\end{aligned}
\end{equation}
\end{proof}

\subsection{Estimates of the reconstructed potential}\label{s45}
For focusing matrix NLS equation, the potential $\boldsymbol{Q}(x)$ can be recovered by the asymptotic formula
\begin{equation}
\boldsymbol{Q}(x)=\lim_{k\rightarrow\infty}2ik\boldsymbol{M}_{12}(x,k).
\end{equation}
From \eqref{mc}, we can rewrite the asymptotic formula as
\begin{align}\label{q1}
\boldsymbol{Q}(x)=\left(-\frac{1}{\pi}\int_{\Gamma}\mu(x,k)\left(V_{x+}(k)-V_{x-}(k)\right)dk\right)_{12}.
\end{align}
However, it is very hard to use \eqref{q1} to directly estimate the potential since the $V_{\pm}-I_{p+q}$ does not vanish at $\pm S_{\infty}$. So next we use Zhou's method \cite{Zhou1989a,Zhou1998} to transform the jump matrix in order to further rewrite the reconstruction formula \eqref{q1}.
\begin{proposition}\label{pw}
When $V_{\pm}-I_{p+q}\in H^{1}(\partial\Omega_{\pm})$, then the matrix function $\eta(k)\in A(\mathbb{C}\setminus\Gamma)$ satisfying the following conditions can be constructed for $l=1$:
\begin{enumerate}[label=(\roman*)]
  \item $\eta_{\pm}\in R(\partial\Omega_{\pm})$ and $\eta_{\pm}-I_{p+q}=\mathcal{O}(k^{-2})$ as $k\rightarrow \infty$.
  \item $\eta_{\pm}$ has the same triangularity as $V_{\pm}$.
  \item $\eta_{\pm}(k)=V_{\pm}(k)+o((k\mp S_{\infty})^{l-1})$.
\end{enumerate}
The notations $A(\mathbb{C}\setminus\Gamma)$ and $R(\partial\Omega_{\pm})$ represent the space of analytic functions in  the region $\mathbb{C}\setminus\Gamma$ and the space of functions whose restrictions on $\partial\Omega_{\pm}$ are rational, respectively.
\end{proposition}
\begin{proof}
We provide a detailed proof using $V_{-}\upharpoonright_{\partial\Omega_{2}}$ as an example, and the
proofs for the other parts are similar. Choose $k_{0}$ in the complement of $\partial\Omega_{2}$ and denote by $f_{\pm S_{\infty}}$ the Taylor polynomial of degree $(l-1)$ of $(k-k_{0})^{n}(V_{-}-I_{p+q})\upharpoonright_{\partial\Omega_{2}}$ at $k=\pm S_{\infty}$. By $V_{-}\upharpoonright_{\partial\Omega_{2}}-I_{p+q}\in H^{1}(\partial\Omega_{2})$ and Appendix I in \cite{Zhou1989a}, there exists $f\in \mathbb{C}_{2l}[k]$ such that
\begin{equation*}
f-f_{\pm S_{\infty}}=o((k\mp S_{\infty})^{l-1}),
\end{equation*}
where $\mathbb{C}_{2l}[k]$ denote the linear space of complex polynomials of degree less than $2l$.
Clearly, $\eta_{-}-I_{p+q}:=(k-k_{0})^{-n}f$ satisfies the conditions (ii) and (iii). Choose $n\geq3$, then the condition (i) holds and $\eta_{-}-I_{p+q}\in H^{1,1}(\partial\Omega_{2})$.
\end{proof}

With the help of the matrix function $\eta$ constructed from Proposition \ref{pw} above, we define a new jump matrix $\mathcal{V}(k)$ by
\begin{align}
\mathcal{V}_{x}(k):=e^{-ixk\operatorname{ad}\sigma_{3}}\mathcal{V}(k)=e^{-ixk\operatorname{ad}\sigma_{3}
}\eta_{-}V(k)\eta_{+}^{-1},
\end{align}
and $\mathcal{V}(k)$ admits a factorization $\mathcal{V}=\mathcal{V}_{-}^{-1}\mathcal{V}_{+}$, where $\mathcal{V}_{-}=V_{-}\eta_{-}^{-1}$ and $\mathcal{V}_{+}=V_{+}\eta_{+}^{-1}$.
For $x\geq0$, if we let $\boldsymbol{M}(x,k)=\mathcal{M}(x,k)e^{-ixk\operatorname{ad}\boldsymbol{\sigma_{3}}}\eta(k)$, then $\mathcal{V}_{x}$ is the jump matrix of the RHP
\begin{equation}
\mathcal{M}_{+}(x,k)=\mathcal{M}_{-}(x,k)\mathcal{V}_{x}(k), \quad k\in\Gamma.
\end{equation}
Clearly, $e^{-ixk\operatorname{ad}\boldsymbol{\sigma_{3}}}\eta\in AL^{\infty}(\mathbb{C}\setminus\Gamma)\cap AL^{2}(\mathbb{C}\setminus\Gamma)$ is deduced from Proposition \ref{pw}, where $AL^{\infty}(\mathbb{C}\setminus\Gamma)\cap AL^{2}(\mathbb{C}\setminus\Gamma)$ represent the space of functions analytic in $\mathbb{C}\setminus\Gamma$ with $L^{\infty}\cap L^{2}$ boundary values.
From $\eta_{\pm}-I_{p+q}=\mathcal{O}(k^{-2})$ as $k\rightarrow\infty$ and the fact that $\boldsymbol{M}$, $\mathcal{M}$ derive the same solution $\mu$ of the associated Beals-Coifman equation, we have
\begin{align}\label{q2}
\begin{aligned}
\boldsymbol{Q}(x)=&\lim_{k\rightarrow\infty}2ik\boldsymbol{M}_{12}(x,k)=\lim_{k\rightarrow\infty}2ik\mathcal{M}_{12}(x,k)\\
=&\left(-\frac{1}{\pi}\int_{\Gamma}\mu(x,k)\left(\mathcal{V}_{x+}(k)-\mathcal{V}_{x-}(k)\right)dk\right)_{12}.
\end{aligned}
\end{align}
The matrix function $\eta$ not only do not contribute to the reconstruction of $\boldsymbol{Q}(x)$, but also makes $\mathcal{V}_{\pm}(k)-I_{p+q}$ vanishes at $\pm S_{\infty}$, i.e.
\begin{equation}\label{vanish}
\mathcal{V}_{\pm}(k)-I_{p+q}=o(1), \quad \text{as} \ k\rightarrow \pm S_{\infty},
\end{equation}
which is very favorable for subsequent estimates.

\begin{figure}
\begin{center}
\begin{tikzpicture}
\draw [fill=pink,ultra thick,color=red!10](-3.5,0) rectangle (3.5,3);
 \draw [fill=pink,ultra thick,color=white!10] (-3.5,0) rectangle (3.5,-2.2);
\filldraw[color=white!10](0,0)-- (2,0) arc (0:180:2);
\filldraw[color=red!10](0,0)-- (2,0) arc (0:-180:2);
\filldraw [color=red!10] (-2,0) -- plot [domain=-2:2,smooth] (\x,{sqrt(-\x*\x/4+1)}) -- (2,0) -- (0,0);
\filldraw [color=white!10] (-2,0) --(2,0) plot [domain=-2:2,smooth] (\x,{-sqrt(-\x*\x/4+1)});
\draw(0,0)circle(2cm);
\draw[dashed][rotate around={0:(0,0)}] (0,0) ellipse (2 and 1);
\draw[fill,black] (2,0) circle [radius=0.04];
\draw[fill,black] (-2,0) circle [radius=0.04];
\draw [ ](-3.5,0)--(3.5,0)  node[right, scale=1] {$\mathbb{R}$};
\draw [-latex](2.95,0)--(3,0);
\draw [-latex](-2.65,0)--(-2.6,0);
\draw [-latex](0,0)--(0.1,0);
\draw [-latex]  (0.15,2)--(0.2,2);
\draw [-latex]  (0.15,-2)--(0.2,-2);
\draw [-latex]  (-0.1,1)--(-0.15,1);
\draw [-latex]  (-0.1,-1)--(-0.15,-1);
\node at (3, 0.4 )  {$+$};
\node at (3, -0.4 )  {$-$};
\node at (-2.8, 0.4 )  {$+$};
\node at (-2.8, -0.4 )  {$-$};
\node at (0, 0.4 )  {$+$};
\node at (0, -0.4 )  {$-$};
\node at (0, 1.4 )  {$-$};
\node at (0, -1.4 )  {$+$};
\node [right] at (-2,-0.3)  {$-S_{\infty}$};
\node [right] at (2,-0.3)  {$S_{\infty}$};
\end{tikzpicture}
\end{center}
\caption{The oriented contour $\widehat{\Gamma}$.}
\label{f3}
 \end{figure}
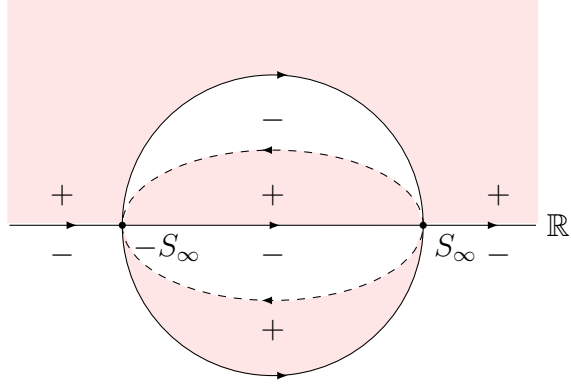

To make full use of the properties of the Cauchy projections on $\mathbb{R}$, the orientation of the line segment $(-S_{\infty},+S_{\infty})$ can reversed by adding the ellipse (dashed) as in Figure \ref{f3}. We denote the new contour as $\widehat{\Gamma}$ and redefine $\mathcal{V}_{\pm}$ as follows:
\begin{enumerate}[label=(\roman*)]
  \item $\mathcal{V}_{\pm}=I_{p+q}$ on the added (dashed) contours.
  \item $\mathcal{V}_{+}(k)$ and $\mathcal{V}_{-}(k)$ are the lower and upper triangular factors in the factorization of $\mathcal{V}$ ($\mathcal{V}^{-1}$) on $\mathbb{R}$ for $|k|>S_{\infty}$ ($|k|<S_{\infty}$).
  \item For $k\in\Gamma_{\infty}$, $\mathcal{V}_{\pm}=I_{p+q}$ for $\text{Im} k\lessgtr0$, and $\mathcal{V}_{\pm}=\mathcal{V}$ for $\text{Im} k\gtrless0$.
\end{enumerate}
The newly defined $\mathcal{V}_{\pm}-I_{p+q}$ is strictly block upper/lower triangular matrix. Moreover, $\mathcal{V}_{\pm}-I_{p+q}$ still belongs to $H^{1,1}(\partial\Omega_{\pm})$ and satisfies \eqref{vanish}. Clearly, we can further rewrite the reconstruction formula \eqref{q2} in the form
\begin{equation}\label{r2}
\boldsymbol{Q}(x)=\left(-\frac{1}{\pi}\int_{\widehat{\Gamma}}\mu(x,k)\left(\mathcal{V}_{x+}(k)-\mathcal{V}_{x-}(k)\right)dk\right)_{12}.
\end{equation}
The associated Beals-Coifman operator is given by
\begin{equation}\label{BC}
\mathcal{C}_{\mathcal{V}_{x}}\phi:=\mathcal{C}_{\widehat{\Gamma}}^{+}\left(\phi\left(I_{p+q}-\mathcal{V}_{x-}\right)\right)+\mathcal{C}_{\widehat{\Gamma}}^{-}\left(\phi(\mathcal{V}_{x+}-I_{p+q})\right),
\end{equation}
where the contour $\widehat{\Gamma}$ can be seen in Figure \ref{f3}. We introduce the notation
\begin{equation}
\Gamma_{\pm}:=\mathbb{R}\cup\{\mathbb{C}^{\pm}\cap\widehat{\Gamma}\}
\end{equation}
for the convenience of subsequent estimates. The definition of the operator $\mathcal{C}_{\widehat{\Gamma}\rightarrow\Gamma'}^{\pm}$ can be found in Zhou \cite{Zhou1998}, where $\Gamma'$ be another curve, especially,
$\mathcal{C}_{\widehat{\Gamma}}^{\pm}=\mathcal{C}_{\widehat{\Gamma}\rightarrow\widehat{\Gamma}}^{\pm}$.
\begin{remark}\label{rem}
The corresponding uniform boundedness of $(I-\mathcal{C}_{\mathcal{V}_{x}})^{-1}$ and the Lipschitz continuity of the map $\mathcal{V}\mapsto\{x\mapsto (I-\mathcal{C}_{\mathcal{V}_{x}})^{-1}\}$ also hold by repeating the same analysis in Lemma \ref{urb}.
\end{remark}
\begin{remark}\label{rem1}
An argument similar to the one used Lemma \ref{lax} shows that the solution $\boldsymbol{M}(x,t,k)$ of the RH problem \ref{RH2} also satisfies the Lax pair \eqref{x1}-\eqref{t1}, where $U(x,t)$ and $P(x,t,k)$ are given by
\begin{equation}
\begin{aligned}
U(x,t)&=-\frac{1}{2\pi}\operatorname{ad}\sigma_{3}\left(\int_{\widehat{\Gamma}}\mu(x,t,k)\left(V_{x+}(t,k)-V_{x-}(t,k)\right)dk\right)\\
&=\begin{bmatrix}0_{p\times p}&\boldsymbol{Q}(x,t)\\ -\boldsymbol{Q}^{\dag}(x,t)&0_{q\times q}\end{bmatrix},
\end{aligned}
\end{equation}
and
\begin{equation}
P(x,t,k)=2kU(x,t)+i\sigma_{3}U_{x}(x,t)-i\sigma_{3}U^{2}(x,t),
\end{equation}
respectively.
\end{remark}
\begin{lemma}
(See \cite{Zhou1998}).
If $x\geq0$, then the jump matrix $\mathcal{V}=\mathcal{V}_{-}^{-1}\mathcal{V}_{+}$ constructed above satisfies
\begin{align}
&\left\|\mathcal{C}_{\widehat{\Gamma}\rightarrow\Gamma_{+}}^{+}(\mathcal{V}_{x-}-I_{p+q})\right\|_{L^{2}(\Gamma_{+})}\leq\frac{c}{\sqrt{1+x^{2}}}\left\|\mathcal{V}_{-}-I_{p+q}\right\|_{H^{1}},\label{11}\\
&\left\|\mathcal{C}_{\widehat{\Gamma}\rightarrow\Gamma_{-}}^{-}(\mathcal{V}_{x+}-I_{p+q})\right\|_{L^{2}(\Gamma_{-})}\leq\frac{c}{\sqrt{1+x^{2}}}\left\|\mathcal{V}_{+}-I_{p+q}\right\|_{H^{1}},\label{22}\\
&\left\|\mathcal{V}_{x-}-I_{p+q}\right\|_{L^{2}(\Gamma_{4})}\leq\frac{c}{\sqrt{1+x^{2}}}\left\|\mathcal{V}_{-}-I_{p+q}\right\|_{H^{1}},\label{33}\\
&\left\|\mathcal{V}_{x+}-I_{p+q}\right\|_{L^{2}(\Gamma_{2})}\leq\frac{c}{\sqrt{1+x^{2}}}\left\|\mathcal{V}_{+}-I_{p+q}\right\|_{H^{1}},\label{44}\\
&\left\|(\mathcal{C}_{\mathcal{V}_{x}})^{2}I_{p+q}\right\|_{L^{2}(\widehat{\Gamma})}\leq\frac{c}{\sqrt{1+x^{2}}}\left\|\mathcal{V}_{-}-I_{p+q}\right\|_{H^{1}}\left\|\mathcal{V}_{+}-I_{p+q}\right\|_{H^{1}}.\label{55}
\end{align}
\end{lemma}
\begin{proposition}\label{vq}
If $\mathcal{V}_{\pm}-I_{p+q}\in H^{1,1}(\partial\Omega_{\pm})$, then $\boldsymbol{Q}\in H^{1,1}(\mathbb{R})$, and the map $(\mathcal{V}_{\pm}-I_{p+q})\mapsto \boldsymbol{Q}$ is Lipschitz continuous from $H^{1,1}(\partial\Omega_{\pm})$ to $H^{1,1}(\mathbb{R})$.
\end{proposition}
\begin{proof}
For $x\geq0$, the integral in the right-hand side of \eqref{r2} can be  rewritten in the form
\begin{equation*}
\int_{\widehat{\Gamma}}\mu(x,k)\left(\mathcal{V}_{x+}(k)-\mathcal{V}_{x-}(k)\right)dk\triangleq F_{1}(x)+F_{2}(x)+F_{3}(x)+F_{4}(x),
\end{equation*}
where
\begin{equation*}
\begin{aligned}
&F_{1}(x)=\int_{\widehat{\Gamma}}\left(\mathcal{V}_{x+}(k)-\mathcal{V}_{x-}(k)\right)dk,\\
&F_{2}(x)=\int_{\widehat{\Gamma}}\left(\mathcal{C}_{\mathcal{V}_{x}}I_{p+q}\right)\left(\mathcal{V}_{x+}(k)-\mathcal{V}_{x-}(k)\right)dk,\\
&F_{3}(x)=\int_{\widehat{\Gamma}}\left((\mathcal{C}_{\mathcal{V}_{x}})^{2}I_{p+q}\right)\left(\mathcal{V}_{x+}(k)-\mathcal{V}_{x-}(k)\right)dk,\\
&F_{4}(x)=\int_{\widehat{\Gamma}}\left(\mathcal{C}_{\mathcal{V}_{x}}(I-\mathcal{C}_{\mathcal{V}_{x}})^{-1}(\mathcal{C}_{\mathcal{V}_{x}})^{2}I_{p+q}\right)\left(\mathcal{V}_{x+}(k)-\mathcal{V}_{x-}(k)\right)dk.
\end{aligned}
\end{equation*}
We can decompose $F_{1}(x)$ into
\begin{equation*}
F_{1}(x)=\int_{\mathbb{R}}\left(\mathcal{V}_{x+}-\mathcal{V}_{x-}\right)dk+\int_{\Gamma_{2}}\left(\mathcal{V}_{x+}-I_{p+q}\right)dk+\int_{\Gamma_{4}}\left(I_{p+q}-\mathcal{V}_{x-}\right)dk,
\end{equation*}
where the first integral belongs to $L^{2,1}(\mathbb{R}^{+})$ by Fourier transform and the second integral makes no contribution to the potential $\boldsymbol{Q}$. Use \eqref{vanish}, and integrate by parts to obtain
\begin{equation*}
\int_{\Gamma_{4}}\left(I_{p+q}-\mathcal{V}_{x-}\right)dk=\frac{1}{2ix}\int_{\Gamma_{4}}e^{-2ixk}\partial_{k}\mathcal{V_{-}}dk.
\end{equation*}
The integral in the right-hand side of the above equality belongs to $L^{2,1}(\mathbb{R}^{+})$ by Laplace transform, hence, $F_{1}\in L^{2,1}(\mathbb{R}^{+})$. Since $F_{2}(x)$ is a block diagonal matrix, the estimate of $F_{2}(x)$ is not needed. We rewrite $F_{3}(x)$ in the following form
\begin{equation*}
\begin{aligned}
F_{3}(x)=\int_{\widehat{\Gamma}}&\left(\mathcal{C}_{\widehat{\Gamma}}^{+}\left(\mathcal{C}_{\widehat{\Gamma}}^{-}\left(\mathcal{V}_{x+}-I_{p+q}\right)\right)\left(I_{p+q}-\mathcal{V}_{x-}\right)\right)\left(\mathcal{V}_{x+}-I_{p+q}\right)\\
&+\left(\mathcal{C}_{\widehat{\Gamma}}^{-}\left(\mathcal{C}_{\widehat{\Gamma}}^{+}\left(I_{p+q}-\mathcal{V}_{x-}\right)\right)\left(\mathcal{V}_{x+}-I_{p+q}\right)\right)\left(I_{p+q}-\mathcal{V}_{x-}\right)dk.
\end{aligned}
\end{equation*}
By the estimates \eqref{fou1}, \eqref{11}, \eqref{33} and Schwarz inequality, we obtain
\begin{align*}
\left|\left(F_{3}\right)_{12}\right|\leq&\left|\left(\int_{\mathbb{R}}+\int_{\Gamma_{4}}\right)\left(\mathcal{C}_{\widehat{\Gamma}}^{-}\left(\mathcal{C}_{\widehat{\Gamma}}^{+}\left(I_{p+q}-\mathcal{V}_{x-}\right)\right)\left(\mathcal{V}_{x+}-I_{p+q}\right)\right)\left(I_{p+q}-\mathcal{V}_{x-}\right)dk\right|\\
\leq&\left|\int_{\mathbb{R}}\left(\mathcal{C}_{\Gamma_{+}\rightarrow\mathbb{R}}^{-}\left(\mathcal{C}_{\widehat{\Gamma}\rightarrow\Gamma_{+}}^{+}\left(I_{p+q}-\mathcal{V}_{x-}\right)\right)\left(\mathcal{V}_{x+}-I_{p+q}\right)\right)\mathcal{C}_{\mathbb{R}}^{+}\left(I_{p+q}-\mathcal{V}_{x-}\right)dk\right|\\
&+\left|\int_{\Gamma_{4}}\left(\mathcal{C}_{\Gamma_{+}\rightarrow\Gamma_{4}}^{-}\left(\mathcal{C}_{\widehat{\Gamma}\rightarrow\Gamma_{+}}^{+}\left(I_{p+q}-\mathcal{V}_{x-}\right)\right)\left(\mathcal{V}_{x+}-I_{p+q}\right)\right)\left(I_{p+q}-\mathcal{V}_{x-}\right)dk\right|\\
\lesssim&\left\|\mathcal{C}_{\widehat{\Gamma}\rightarrow\Gamma_{+}}^{+}\left(I_{p+q}-\mathcal{V}_{x-}\right)\right\|_{L^{2}(\Gamma_{+})}\left\|\mathcal{C}_{\mathbb{R}}^{+}\left(I_{p+q}-\mathcal{V}_{x-}\right)\right\|_{L^{2}(\mathbb{R})}\\
&+\left\|\mathcal{C}_{\widehat{\Gamma}\rightarrow\Gamma_{+}}^{+}\left(I_{p+q}-\mathcal{V}_{x-}\right)\right\|_{L^{2}(\Gamma_{+})}\left\|\left(I_{p+q}-\mathcal{V}_{x-}\right)\right\|_{L^{2}(\Gamma_{4})}\\
\lesssim&\frac{1}{1+x^{2}}.
\end{align*}
For $F_{4}(x)$, we set $h=(I-\mathcal{C}_{\mathcal{V}_{x}})^{-1}(\mathcal{C}_{\mathcal{V}_{x}})^{2}I_{p+q}$. It follows from Remark \ref{rem} and \eqref{55} that
\begin{equation}\label{h}
\|h\|_{L^{2}(\widehat{\Gamma})}\leq\left\|(I-\mathcal{C}_{\mathcal{V}_{x}})^{-1}\right\|_{L^{2}\rightarrow L^{2}}\left\|(\mathcal{C}_{\mathcal{V}_{x}})^{2}I_{p+q}\right\|_{L^{2}}
\lesssim\frac{1}{\sqrt{1+x^{2}}}.
\end{equation}
By the inequalities \eqref{fou1}, \eqref{33} and \eqref{h}, we can use Schwarz inequality to estimate
\begin{align*}
\left|\left(F_{4}\right)_{12}\right|\leq&\left|\int_{\widehat{\Gamma}}\left(\mathcal{C}_{\widehat{\Gamma}}^{-}h\left(\mathcal{V}_{x+}-I_{p+q}\right)\right)\left(I_{p+q}-\mathcal{V}_{x-}\right)dk\right|\\
\leq&\left|\int_{\mathbb{R}}\left(\mathcal{C}_{\widehat{\Gamma}\rightarrow\mathbb{R}}^{-}h\left(\mathcal{V}_{x+}-I_{p+q}\right)\right)\mathcal{C}_{\mathbb{R}}^{+}\left(I_{p+q}-\mathcal{V}_{x-}\right)dk\right|\\
&+\left|\int_{\Gamma_{4}}\left(\mathcal{C}_{\widehat{\Gamma}\rightarrow\Gamma_{4}}^{-}h\left(\mathcal{V}_{x+}-I_{p+q}\right)\right)\left(I_{p+q}-\mathcal{V}_{x-}\right)dk\right|\\
\lesssim&\left\|h\right\|_{L^{2}(\widehat{\Gamma})}\left(\left\|\mathcal{C}_{\mathbb{R}}^{+}\left(I_{p+q}-\mathcal{V}_{x-}\right)\right\|_{L^{2}(\mathbb{R})}+\left\|I_{p+q}-\mathcal{V}_{x-}\right\|_{L^{2}(\Gamma_{4})}\right)\\
\lesssim&\frac{1}{1+x^{2}},
\end{align*}
which completes the proof of $\boldsymbol{Q}\in L^{2,1}(\mathbb{R}^{+})$.

Proceeding as in the proof of Proposition \ref{H11}, we have $\boldsymbol{Q}\in H^{1}(\mathbb{R}^{+})$ for focusing matrix NLS equation. For $x\leq0$, we can use a similar manner to obtain the corresponding estimate by considering $\widetilde{\boldsymbol{M}}(x,k)$ that is associated with $\widetilde{\boldsymbol{Q}}_{x_{0}}$ and satisfies \eqref{tim}. Finally, it follows from Remark \ref{rem} that the map is Lipschitz continuous.
\end{proof}

\subsection{Proof of Theorem \ref{thm} for focusing case}\label{s46}
Based on the results of Sections \ref{s2} and \ref{s4}, we can prove the
$L^{2}$-Sobolev space bijectivity of the inverse scattering
and existence of global solutions to  the Cauchy problem \eqref{mnls}-\eqref{mnls1} for focusing   matrix NLS equation.
\begin{proof}
The results of Propositions \ref{pvpm} and \ref{vq} indicate that the maps \eqref{D22} and \eqref{I22} are Lipschitz continuous.
Given initial date $\boldsymbol{Q}_{0}\in H^{1,1}(\mathbb{R})$, Proposition \ref{ct} shows that the map $\boldsymbol{Q}_{0}\mapsto V_{\pm}(t,k)-I_{p+q}$ is Lipschitz continuous from $H^{1,1}(\mathbb{R})$ to $C([-T,T], H^{1,1}(\partial\Omega_{\pm}))$ for every $T>0$. Proposition \ref{prh2} and Remark \ref{rem1} show that the $\boldsymbol{Q}(x,t)$ defined by \eqref{r2} solves \eqref{mnls}. According to the results of Proposition \ref{vq}, we have the map $\boldsymbol{Q}_{0}\mapsto \boldsymbol{Q}$ is Lipschitz continuous from $H^{1,1}(\mathbb{R})$ to $C\left([-T,T], H^{1,1}(\mathbb{R})\right)$ for every $T>0$. Finally, we can use a similar way as defocusing case to prove that the uniqueness of the solution.
\end{proof}

\section*{Acknowledgements}
\addcontentsline{toc}{section}{Acknowledgements}
This work is supported by the National Natural Science Foundation of China   (Grant No. 12271104, 51879045). \vspace{2mm}
	
	\noindent\textbf{Data Availability Statements}
	
	The data that supports the findings of this study are available within the article.\vspace{2mm}
	
	\noindent{\bf Conflict of Interest}
	
	The authors have no conflicts to disclose.

\end{document}